\documentclass[10pt,a4paper]{article}
\usepackage{graphicx,epstopdf} 
\usepackage[caption=false]{subfig} 
\usepackage{hyperref}
\usepackage{graphicx,color}
\usepackage{amsmath,amssymb}
\usepackage{amsthm}
\usepackage{fullpage}
\usepackage{mathabx}
\usepackage[capitalize, nameinlink]{cleveref}
\usepackage{biblatex}
 \newtheorem{theorem}{Theorem}[section]
 \newtheorem{lemma}[theorem]{Lemma}
 \newtheorem{proposition}[theorem]{Proposition}

 \newtheorem{definition}{Definition}
 \newtheorem{example}{Example}

 \newtheorem{assumption}{Assumption}

\crefname{section}{section}{sections}
\crefname{subsection}{subsection}{subsections}
\Crefname{section}{Section}{Sections}
\Crefname{subsection}{Subsection}{Subsections}

\Crefname{figure}{Figure}{Figures}

\crefformat{equation}{\textup{#2(#1)#3}}
\crefrangeformat{equation}{\textup{#3(#1)#4--#5(#2)#6}}
\crefmultiformat{equation}{\textup{#2(#1)#3}}{ and \textup{#2(#1)#3}}
{, \textup{#2(#1)#3}}{, and \textup{#2(#1)#3}}
\crefrangemultiformat{equation}{\textup{#3(#1)#4--#5(#2)#6}}%
{ and \textup{#3(#1)#4--#5(#2)#6}}{, \textup{#3(#1)#4--#5(#2)#6}}{, and \textup{#3(#1)#4--#5(#2)#6}}

\Crefformat{equation}{#2Equation~\textup{(#1)}#3}
\Crefrangeformat{equation}{Equations~\textup{#3(#1)#4--#5(#2)#6}}
\Crefmultiformat{equation}{Equations~\textup{#2(#1)#3}}{ and \textup{#2(#1)#3}}
{, \textup{#2(#1)#3}}{, and \textup{#2(#1)#3}}
\Crefrangemultiformat{equation}{Equations~\textup{#3(#1)#4--#5(#2)#6}}%
{ and \textup{#3(#1)#4--#5(#2)#6}}{, \textup{#3(#1)#4--#5(#2)#6}}{, and \textup{#3(#1)#4--#5(#2)#6}}

\crefdefaultlabelformat{#2\textup{#1}#3}

\crefname{assumption}{Assumption}{Assumptions}
\crefname{example}{Example}{Examples}

\newcommand{\real}{\mathbb{R}} 
\newcommand{\zDomain}{\mathcal{Z}}
\newcommand{\Domain}{\mathcal{X}\times\mathcal{Y}}
\newcommand{\realDomain}{\mathbb{R}^{n}\times\mathbb{R}^{m}}
\newcommand{\argmax}[2] {\mathrm{arg}\max_{#1}#2}
\newcommand{\argmin}[2] {\mathrm{arg}\min_{#1}#2}
\newcommand{\norm}[1]{\Vert #1 \Vert}
\DeclareMathOperator{\G}{\mathcal{G}}

\DeclareMathOperator{\EX}{\mathbb{E}}
\DeclareMathOperator{\sW}{subW}

\DeclareMathOperator{\diag}{diag}
\DeclareMathOperator{\dist}{d}

\title{Stochastic Saddle Point Problems with Decision-Dependent Distributions}
\date{}
\author{Killian Wood\thanks{Department of Applied Mathematics, University of Colorado, Bouder, CO (killian.wood@colorado.edu).} 
\ and Emiliano Dall'Anese\thanks{Department of Electrical, Computer, and Energy Engineering, and Department of Applied Mathematics, University of Colorado, Bouder, CO (emiliano.dallanese@colorado.edu).}
}

\begin{document}

\maketitle
\begin{abstract} This paper focuses on stochastic saddle point problems with decision-dependent distributions. These are problems whose objective is the expected value of a stochastic payoff function and whose data distribution drifts in response to decision variables\textemdash a phenomenon represented by a distributional map. A common approach to accommodating distributional shift is to retrain optimal decisions once a new distribution is revealed, or repeated retraining. We introduce the notion of equilibrium points, which are the fixed points of this repeated retraining procedure, and provide sufficient conditions for their existence and uniqueness. To find equilibrium points, we develop deterministic and stochastic primal-dual algorithms and demonstrate their convergence with constant step-size in the former and polynomial decay step-size schedule in the latter. By modeling errors emerging from a stochastic gradient estimator as sub-Weibull random variables, we provide error bounds in expectation and in high probability that hold for each iteration. Without additional knowledge of the distributional map, computing saddle points is intractable. Thus we propose a condition on the distributional map\textemdash which we call opposing mixture dominance\textemdash that ensures that the objective is strongly-convex-strongly-concave. Finally, we demonstrate that derivative-free algorithms with a single function evaluation are capable of approximating saddle points.
\end{abstract}

\section{Introduction}
The broad goal of stochastic optimization is to find an optimal decision for an objective with uncertainty in some parameters~\cite{nemirovski2009robust,shapiro2005complexity,zhang2021generalization}. For example, in statistical learning, parameters may be taken to be data-label pairs in large data-sets~\cite{simsekli2019tail,gurbuzbalaban2020heavy}; in the context of optimization of physical and dynamical systems, they may model externalities and random exogenous inputs, or system parameters that are predicted from data and are accompanied by given error statistics~\cite{bianchin2021online}. A key assumption that is typically leveraged for providing  theoretical guarantees for stochastic optimization algorithms is that the distributions of random parameters are stationary~\cite{birge2011introduction}. However, in modern machine learning and cyber-physical systems applications,  data may be subject to decision-dependent shift, whereby the distribution is inextricably tied to the decision variables.

We are interested in solving a stochastic saddle point problem where the data distribution shifts in response to decision variables. This feature yields the problem: 
\begin{equation}
\label{problemStatement}
    \min_{x\in\mathcal{X}}\max_{y\in\mathcal{Y}} ~ \left\{\Phi(x,y): = \underset{w\sim D(x,y)}{\EX}[\phi(x,y,w)]\right\},
\end{equation}
where $\mathcal{X} \subset \mathbb{R}^n$ and $\mathcal{Y} \subset \mathbb{R}^n$ are compact constraint sets, $\phi:\realDomain\rightarrow\real$ is a scalar-valued function of the decision variables $(x, y)$ parameterized by a random vector $w$, $D$ is a distribution inducing map, and $w$ is supported on a complete and separable metric space $M$ with metric $\dist$. Hereafter, we refer to $\Phi$ as the objective and the function $\phi$ as the minimax function. We remark that the distribution of $w$ depends on the decision variables $(x,y)$. When solutions to the problem~\cref{problemStatement} exist, we will denote these solutions as $(x^{*},y^{*})$. 

For general distributional maps $D$, solving \cref{problemStatement} directly is intractable. Indeed, $\Phi$ may be non-convex-non-concave even when $\phi$ is strongly-convex-strongly-concave. A common heuristic when dealing with non-stationary data distributions is to recompute optimal decisions each time a new data distribution is revealed. For minimax problems, this corresponds to generating a sequence of decisions $\{(x_{t},y_{t})\}_{t\geq 0}$ such that:
\begin{equation}
\begin{aligned}
\label{eqn:repeated_retraining}
   & x_{t+1} \in \argmin{x\in\mathcal{X}}\max_{y\in\mathcal{Y}} \underset{w\sim D(x_{t},y_{t})}{\EX} [\phi(x,y,w)] \\ 
   & y_{t+1} \in \argmax{y\in\mathcal{Y}}\min_{x\in\mathcal{X}} \underset{w\sim D(x_{t},y_{t})}{\EX} [ \phi(x,y,w)].
\end{aligned}
\end{equation}
We will refer to fixed points of this sequence as \textit{equilibrium points}. These can be seen as the counterparts of the so-called performatively stable points in~\cite{drusvyatskiy2022stochastic,perdomo2020performative,wood2021online} in our stochastic minimax setup \cref{problemStatement}.
A primary objective of this work is illustrating sufficient conditions for the existence and uniqueness of equilibrium points. In particular, existence of the set equilibrium points is shown when the minimax function is convex in $x$ and concave in $y$ for a given $w$, and under continuity of the distributional map. Building on these results, and focusing on strongly-convex-strongly-concave functions $\phi$, we then develop deterministic and stochastic projected primal-dual algorithms that can determine equilibrium points. 

However, as discussed in the paper, equilibrium points and saddle points are qualitatively distinct. Equilibrium points are saddle points for the stationary problem that they induce, but need not be necessarily optimal. For this reason, we investigate a sufficient condition on the distributional map $D$ that allows us to guarantee strong-convexity-strong-concavity of the objective $\Phi$. We call this condition \textit{opposing mixture dominance}, and provide a detailed example of a practical class of distributions that satisfy this assumption. Since gradient based algorithms will require us to have knowledge of the explicit dependence $D$ has on the decision variables, we turn to zeroth order algorithms. We show that demonstrate that derivative-free algorithms with a single function evaluation are capable of approximating saddle points provided that $\Phi$ is strongly-convex-strongly-concave. Belowe we provide additional notational conveiences that will be used throughout the work.\footnote{\textit{\textbf{Notation}}. We let $\real$
denote the set of real numbers, and $[n]=\{1,2,\ldots,n\}$. For given column  vectors $x \in \mathbb{R}^n$ and $y \in \mathbb{R}^n$, we let $(x, y)$ denote their concatenation; that is, $[x^T, y^T]^T$, with $^T$ denoting transposition; for $x,y \in \mathbb{R}^n$,  we let $\langle x, y\rangle$ denote the inner product. For a given column vector $x \in \mathbb{R}^n$, $\| x\|$ is the Euclidean norm; for a matrix $X \in \mathbb{R}^{n \times m}$, $\|X\|_F$ denotes the Frobenious norm and $\|X\|_\star$ the nuclear norm. Given a differentiable function $f: \mathbb{R}^n \rightarrow \mathbb{R}$, $\nabla f(x)$ denotes the gradient of $f$ at $x$ (taken to be a column vector). For a function $f: \mathbb{R}^n \times \mathbb{R}^m \rightarrow \mathbb{R}$, $\nabla_x f(x,x)$ denotes the partial derivatives of $f$ with respect to $x$. Given a closed convex set $C \subseteq \mathbb{R}^n$, $\Pi_{C}:\mathbb{R}^n \to \mathbb{R}^n$ denotes the Euclidean projection of $y$ onto $C$, namely $\Pi_{C} (y) := \arg \min_{v \in C} \norm{y-v}$.  Given a set $C \subseteq \mathbb{R}^n$, $P(C)$ denotes its power set. For a given random variable $w \in \mathbb{R}$, we write $w\sim\mu$ to mean that $w$ is a random variable with law $\mu$, a probability measure supported on $\real$. Hence $\mu(A)=\mathbb{P}(w\in A)$ for all $A\subseteq\real$. Furthermore, $\mathbb{E}[w]$ denotes the expected value of $w$, and $\mathbb{P}(w \leq \epsilon)$ denotes the probability of $w$ taking values smaller than or equal to $\epsilon$; $\|w\|_p := \mathbb{E}[|w|^p]^{1 / p}$, for any $p \geq 1$.}

\subsection{Motivation}

Saddle-point problems arise in a variety of  areas, including stochastic constrained optimization problems~\cite{le2012adaptive,ribeiro2010ergodic}, strategic classification~\cite{perdomo2020performative}, and  games~\cite{facchinei2014non}, and are relevant in several applications that span finance, energy systems, transportation networks, and ride-sharing just to mention a few. Our work naturally contributes modeling and algorithmic approaches in these areas and applications whenever the uncertainty can be considered to be decision dependent. For example, in the energy systems context, the problem of finding an optimal charging policy for a fleet of electric vehicles involves uncertainty in the price of energy; indeed, real-time prices are subject to change 
due to the demand itself as well as on external factors such as spot market behavior~\cite{fele2019scenario}.  As representative examples, in the following we provide a brief description of the problem formulations in competitive markets and strategic classification. 

\subsubsection{Relative cost maximization in competitive markets}
\label{subsec:electric_vehicles}
Consider a game in which two competing service providers aim to maximize their relative profits in a region partitioned in $n$ zones. This applies to, for example, ride sharing \cite{bianchin2021online} and power providers \cite{archarya2009competition}. Focusing on electric vehicle charging station providers \cite{li2021electric}, at each zone $i\in[n]$ we denote the average baseline price as $p_{i}$ and the price differential to charge per-minute set by provider one as $x_{i}$. The revenue of provider one is $a_{i}(x_{i}+p_{i})$, based on their demand $a_{i}$. However, they must incorporate a zone based utility cost $\theta_{i}(x_{i}+p_{i})$, as well as well as a term enforcing quality of service $\gamma_{1,i}x_{i}^2$ (the quadratic term balances the utility of the provider with the cost of ensuring quality of service by penalizing large deviations from the baseline price). 

In total, the profit for provider one over all $n$ zones is given by $u_{1}(x,a) = \langle a+\theta ,x + p\rangle -\Vert \Gamma_{1} x \Vert^{2}$, 
with $\Gamma_{1} = \diag\{ \gamma_{1,1},\dots,\gamma_{1,n} \}$. If the price and demand of service for provider two are given by $y$ and $b$ respectively, then their profit is similarly represented as $u_{2}(y,b) = \langle b+\theta ,y+p\rangle -\Vert \Gamma_{2} y \Vert^{2}$. Each provider has finite bounds on the prices they are willing to set in each zone, and hence their prices are constrained to the closed rectangles $\mathcal{X}=\bigtimes_{i=1}^{n} [-p_{i},c_{1,i}p_{i}]$ and $\mathcal{Y}=\bigtimes_{i=1}^{n} [-p_{i},c_{2,i}p_{i}]$ with multiplicative factors $c_{j,i}>0$. The service demand vectors $a$ and $b$ are unknown quantities that will depend not only on the price set by their respective providers, but also their competition. One such example of a dependence is a best response model. It has been shown that best response models with linear utility and quadratic cost associated with changing features give rise to location-scale models of the form:
    $a \overset{d}{=} a_{0} + A_{1}x + A_{2}y$,
    $b \overset{d}{=} b_{0} + B_{1}x + B_{2}y$,
where $a_{0}\sim D_{1}$ and $b_{0}\sim D_{2}$, for which $D_{1}$ and $D_{2}$ represent stationary prior distributions for the demand associated with providers one and two respectively~\cite{perdomo2020performative}. In order to maximize their expected profit relative to provider two, provider one will minimize the negative of their relative profit given by $u_{1}(x,a)-u_{2}(y,b)$, and hence the optimal strategies for both providers are solutions to the saddle point problem 
\begin{equation}
\label{application:EV}
    \min_{x\in\mathcal{X}}\max_{y\in\mathcal{Y}} \underset{(a,b)\sim D(x,y)}{\EX} 
    \Vert \Gamma_{1}x \Vert^{2} - \Vert \Gamma_{2}y \Vert^{2}  
    - \langle a + \theta, x \rangle + \langle b + \theta, y \rangle,
\end{equation}
where the dependence on baseline price $p$ has been removed as it has no impact on the optimality criterion.

\subsubsection{Multitask Strategic Classification}
\label{subsec:mutltitask}
A second application is a multitask strategic classification problem \cite{perdomo2020performative}. Consider the problem of learning a collection of $N$ tasks for a strategic population simultaneously. Such a framework is especially helpful in problems such as spam filtering; in this case, a classifier is learned for each user, however it is expected that spam email designers will adapt to spam filters in order to be miss-classified \cite{online_multitask_learning,spam_filter_bayes,weinberger2009feature}. Each user has a dataset $\mathcal{D}_{i} = \{(a_{i,j},b_{i,j}\}_{j=1}^{n_{i}}$ with features from their emails and labels of either spam or not spam. The goal is to learn classifiers $\{f_{i}\}_{i=1}^{N}$ that will predict spam based not only on the available data, but on the anticipated changes spammers will make to their features. To learn $\{f_{i}\}_{i=1}^{N}$, we pose the nuclear norm regularized minimization problem
\begin{equation}
\label{mutltitask_regularized}
    \min_{x_{1},\dots,x_{N}\in\mathcal{X}} \quad \sum_{i=1}^{N} \ \underset{(a,b)\sim D_{i}(x_{i})}{\EX}
    \ell(b, f_{i}(a;x_{i})) + \Vert [x_{1},\dots,x_{N}]\Vert_{*},
\end{equation}

A challenge in solving~\eqref{mutltitask_regularized} is that the nuclear norm is not differentiable. To avoid using proximal methods, we can introduce dual variables by observing that the nuclear norm is the dual of the operator norm and hence
$\norm{X}_{\star} = \max_{Y^{T}Y\preceq I }  -\langle X,Y \rangle_{F}$, 
where $\langle\cdot,\cdot\rangle_{F}:\real^{d\times m}\times\real^{d\times m}\mapsto \real$ is the Frobenius inner product defined by $\langle X,Y\rangle_{F} = \text{trace}(X^{T}Y)$~\cite{pong2010trace}. Thus, the problem in \eqref{mutltitask_regularized} can be rewritten as 
\begin{equation}
    \min_{X\in\mathcal{X}^{T}}\max_{Y\in\mathcal{Y}} \quad \sum_{t=1}^{T} \ \underset{(a,b)\sim D_{t}(x_{t})}{\EX}
    \ell(b, f_{t}(a;x_{t})) -\langle X,Y \rangle_{F},
\end{equation}
where $\mathcal{X}^{N}=\bigtimes_{i=1}^{N} \mathcal{X}\subseteq \real^{d\times N}$, and $\mathcal{Y}=\{ Y \vert \ Y^{T}Y\preceq I\}$.

\subsection{Related Works}

In this subsection, we review the literature on saddle-point problems and stochastic optimization that is most relevant to our work.   

\emph{Saddle Point Problems}. Saddle point problems have been well investigated, with theoretical guarantees in distance metrics (squared Euclidean metric) and dual-gap metrics. Common approaches for these problems are either proximal methods or primal-dual methods, with version existing for both deterministic and stochastic problems. Proximal methods include Mirror-Prox \cite{nemirovski2004prox} and  Extragradient \cite{mokhtari2020unified} for deterministic objectives and Stochastic Mirror-Prox \cite{nemirovski2009robust} for stochastic objectives; primal-dual methods include Primal-Dual or Gradient Descent-Ascent \cite{koshal2011multiuser} and Optimistic Gradient Descent-Ascent \cite{mokhtari2020unified}. Stochastic Primal-Dual methods and their accelerated varieties are studied in \cite{zhang2021robust}. Convergence guarantees for stochastic optimization algorithms typically come in the form of convergence with a diminishing step-size policy \cite{natole2018stochastic,nemirovski2004prox,nemirovski2009robust}, or convergence to a neighborhood with fixed step-size \cite{koshal2011multiuser,zhang2021robust}.

\emph{Decision-Dependent Distributions}. Our work is most closely related to the literature on stochastic optimization with decision-dependent distributions---also referred to as performative prediction in the machine learning community. These are two related paradigms for stochastic programs where the data distribution is a function of the decision variables. 
In this context,  equilibrium problems are presented and solved via conceptual deterministic algorithms in \cite{perdomo2020performative}; second moment analysis for stochastic algorithms with access to a sampling oracle is provided in \cite{mendler2020stochastic}; and proximal first-order algorithms for regularized objectives are studied in \cite{drusvyatskiy2022stochastic}. In \cite{wood2021online}, first moment and high probability tracking analysis are provided under a sub-Weibull gradient error for a time-varying optimization problem.

\emph{Sub-Weibull Error Models}. Sub-Gaussian and sub-exponential gradient error models are common in the literature on stochastic gradient methods. Empirical and theoretical results have demonstrated that heavier tailed distributions arise naturally in deep learning. The class of sub-Weibull random variables subsumes the sub-Gaussian and sub-Exponential classes of distributions, while also including heavier tailed distributions and random variables with bounded support \cite{kuchibhotla2018moving,vladimirova2020sub,wong2020lasso}.  Given their broad application, this model has been receiving increasing use in the literature. See for example,~\cite{kuchibhotla2018moving,bastianello2021stochastic,wood2021online}. 

\subsection{Contributions}
\label{sec:contributions}
In this paper, we offer the following main contributions.

\noindent \emph{(1)} \textit{The Minimax Equilibrium Problem}. We propose a notion of equilibrium points for the saddle point problem in \cref{problemStatement}. We then provide conditions to guarantee their existence and uniqueness, and we provide bounds for the distance between the unique equilibrium point and saddle points of \eqref{problemStatement}.

\noindent \emph{(2)} \textit{Algorithms}. We demonstrate that primal-dual algorithms, using the gradients of $\phi$, are effective at finding equilibrium points when the stochastic objective $\phi$ is strongly-convex-strongly-concave. First, we demonstrate convergence of a conceptual algorithm using full gradient information. We then demonstrate that stochastic algorithms with fixed step-size converge to a noise-dominated neighborhood of the equilibrium point, as well as provide expectation bounds and high probability bounds that hold for each iteration. Additionally, we show convergence of the stochastic algorithm for decaying step-size. 

\noindent \emph{(3)} \textit{Saddle Points}. We propose a sufficient condition for distributional maps that preserve strong-convexity-strong-concavity of $\phi$. Hence when $\phi$ is strongly-convex-strongly-concave, so too is $\Phi$. We then discuss a zeroth-order algorithm capable of finding an approximate saddle point using only a single function evaluation. 

\noindent \emph{(4)} \textit{Experiments}. We illustrate our results on the electric vehicle charging problem in \cref{application:EV} by incorporating synthetic demand data from \cite{gilleran2021electric} for a location-scale family based distributional map. 

\emph{Connection to Related Works}.
Relative to the referenced work on stochastic saddle-point problems, we consider the case where the function $\phi(x, y, w)$ is strongly-convex-strongly-concave in the decisions $(x,y)$ and the randomness captured by $w$ is governed by a family of distributions $(x,y)\mapsto D(x,y)$. To measure performance of our algorithms we use the Euclidean distance to the solution. Contrary to most of the work on the stochastic saddle point problems that use the distance metric, we analyze primarily using both the first moment and the second moment. We note that, however, bounds on the first moment do not need  stochastic filtrations as an underlying working assumption. We also provide bounds in high probability that hold for each iteration, under a sub-Weibull model. 

The prior body of work on optimization with decision dependent distribution has studied both equilibrium points and minimizers, as well as stochastic algorithms for finding them. Hence, our result extend those found in the setting of \cite{drusvyatskiy2022stochastic,mendler2020stochastic,miller2021outside,perdomo2020performative} by casting the problem into a more general variational-inequality framework. The study of saddle point problems is a unique contribution of this paper. Relative to \cite{miller2021outside}, we continue the study of stochastic orders and location-scale families by proposing a condition on the distributional map that is suitable for convex-concave objectives. To approximate saddle points, we analyze a zeroth order algorithm adapted from works such as \cite{ bravo2018bandit,drusvyatskiy2021improved,miller2021outside} to suit our setting.

\emph{Organization}. The paper is organized as follows. In Section 2 we define and solve the the equilibrium point problem in the static setting. Section 3 investigates the time-varying equilibrium point problem via online optimization methods. Section 4 offers sufficient conditions for which the saddle point problem can be solved. In Section 5, we illustrate our results on a competitive market with a demand-response price model.

 \section{The Equilibrium Problem}
\label{sec:equilibriuim}
Recall the decision-dependent stochastic saddle point problem provided in \cref{problemStatement}:
\begin{equation}
    \label{tiProblem}
        \min_{x\in\mathcal{X}}\max_{y\in\mathcal{Y}} \left\{\Phi(x,y): = \underset{w\sim D(x,y)}{\EX}[\phi(x,y,w)]\right\}
\end{equation}
where the sets $\mathcal{X} \subset \mathbb{R}^n$ and $\mathcal{Y} \subset \mathbb{R}^m$ are convex and compact. Let $\mathcal{P}(M)$ be the set of  Radon probability measures on a complete and separable metric space $M$ with finite first moment, and observe that the objective function can be written in integral form as 
     $\Phi(x,y) = \int_{M} \phi(x,y,w) \mu_{(x,y)}\left(dw\right)$
where $\mu_{(x,y)}\in\mathcal{P}(M)$ is given as the output of the distributional map $D$ for each $(x,y)\in\zDomain$. Classical solutions to this problem take the form of saddle points, as defined next. 

\begin{definition}{\textbf{\textit{(Saddle Points)}}} 
A pair $(x^{*},y^{*})\in\Domain$ is a saddle point for the problem in \eqref{problemStatement} provided that $\Phi(x^{*},y) \leq \Phi(x^{*},y^{*}) \leq \Phi(x,y^{*})$, 
$\forall \, x\in\mathcal{X}, y\in\mathcal{Y}$. 
\end{definition}

Sufficient conditions for the existence of saddle points consist of $\Phi$ being convex-concave while $\mathcal{X}$ and $\mathcal{Y}$ are convex and compact \cite[Ex. 11.52]{rockafellar2009variational}. When minimax equality holds, we can equivalently characterize saddle points as a pair that satisfies:
\begin{equation*}
     x^{*}\in\argmin{x\in\mathcal{X}} \max_{y\in\mathcal{Y}} \Phi(x,y), \quad
    y^{*}\in\argmax{y\in\mathcal{Y}} \min_{x\in\mathcal{X}} \Phi(x,y).
\end{equation*}
In practice, computing saddle points directly is computationally intractable. Namely, the dependence of the distributional map on the decision variables implies that even when $\phi$ is convex-concave $\Phi$ may not be and hence saddle points will not even exist. Hence, we direct our attention to the fixed point of the repeated retraining heuristic in \cref{eqn:repeated_retraining}. 

\begin{definition}{\textbf{\textit{(Equilibrium Points)}}} 
\label{def:equilibrium}
A pair $(\bar{x},\bar{y})\in\Domain$ is an equilibrium point if:
\begin{equation}
    \begin{aligned}
    & \bar{x}\in \arg\min_{x\in\mathcal{X}} \left\{\max_{y\in\mathcal{Y}} \underset{w\sim D(\bar{x},\bar{y})}{\EX} [\phi(x,y,w)] \right\}, \\
    & \bar{y}\in \arg\max_{y\in\mathcal{Y}}\left\{\min_{x\in\mathcal{X}} \underset{w\sim D(\bar{x},\bar{y})}{\EX} [ \phi(x,y,w)] \right\}.
    \end{aligned}
\end{equation}
\end{definition}
Intuitively, $(\bar{x},\bar{y})$ are saddle points for the stationary saddle point problem induced by the distribution $D(\bar{x},\bar{y})$. 
These are desirable as alternative solutions as they exist under mild convexity assumptions for problems with compact decision sets. Furthermore, we note that compactness here is not a limitation, as even unconstrained problems can be artificially constrained to a sufficiently large compact set without changing the solutions \cite{koshal2011multiuser}.

Our first objective in this work will be to provide conditions for the existence and uniqueness of these equilibrium points. Later, we develop first order algorithms and demonstrate their convergence to equilibrium points. Crucial to our analysis will be the properties of the ``decoupled objective,'' which is defined as $\Phi(x,y;x',y') = \EX_{w\sim D(x',y')} [\phi(x,y,w)]$
for $x,x'\in\mathcal{X}$ and $y,y'\in\mathcal{Y}$. Here, the distribution is fixed, for given points $(x',y')$. With these definitions, we consider a correspondence $H:\Domain\rightarrow P(\Domain)$, defined by 
\begin{equation}
\label{def:eqMap}
    H(x,y) = \left( \arg\min_{x'\in\mathcal{X}}\max_{y'\in\mathcal{Y}}\Phi(x',y';x,y), ~ \arg\max_{y'\in\mathcal{Y}}\min_{x'\in\mathcal{X}}  \Phi(x',y';x,y)\right)
\end{equation}
which maps pairs in the product space to its power set $P(\Domain)$. 
In light of \cref{def:equilibrium}, the equilibrium points are fixed points of the map $H$; that is,  
$(\bar{x},\bar{y})\in H(\bar{x},\bar{y}).$ 

For notional convenience, we will introduce the stacked vector in the product space $z=(x,y)\in\Domain$ (consequently, we can identify $H(z)$ and $\Phi(z';z)$ with the above functions whenever convenient). In the following section, we provide sufficient conditions for the existence of equilibrium points. 

\subsection{Existence of Equilibrium Points}
Our goal is to demonstrate the existence and uniqueness of equilibrium points. First, we demonstrate the existence of equilibrium points by showing that the fixed point set of $H$, defined as $\text{Fix}(H):=\{z\in\Domain\vert ~ z\in H(z)\}$, is nonempty. 
The crux of our proof is showing that, under appropriate assumptions, $H$ is an upper hemicontinous function. Next, we provide this definition as well as the notion of a topological neighborhood. 

\begin{definition}{(\textit{\textbf{Neighborhood)}}}\cite[Sec. 17.2]{guide2006infinite}
If A is a topological space and $x\in A$ , then a \textit{neighborhood} of $x$ is a set $V\subset A$ such that there exists an open set $U$ with $x\in U \subset V$. 
If the set $V$ is open, then we say that $V$ is an \textit{open neighborhood}. 
\end{definition}

\begin{definition}{\textit{\textbf{(Upper Hemicontinuity)}}}\cite[Sec. 17.2]{guide2006infinite}
If $A$ and $B$ are two topological metric spaces, then a set valued function $\varphi:A\mapsto P(B)$ is \textit{upper hemicontinuous} (uhc) at $x\in A$ provided that for every neighborhood $U$ of $\varphi(x)\subset B$, the upper inverse set 
$\varphi^{u}(U) = \{x:\varphi(x)\subset U \}$ is a neighborhood of $x$. If $\varphi$ is uhc at every $x$ in $A$, then we say that $\varphi$ is uhc on A. 
\end{definition}

We next state our result for the existence of equilibrium points. 

\begin{theorem}{(\textit{\textbf{Existence of Equilibrium Points}})}
\label{thm:existence}
Suppose that the following assumptions hold:  

\noindent i) $x \mapsto \phi(x,y,w)$ is convex in $x$ for all $y\in\mathcal{Y}$ and for all realizations of $w$; 

\noindent ii) $y \mapsto \phi(x,y,w)$ is concave in $y$ for all $x\in\mathcal{X}$ and for all realizations of $w$; 

\noindent iii) $\phi$ is continuous on $\Domain$ for all $w$; 

\noindent iv) $\mathcal{X}\subset\real^{d},\mathcal{Y}\subset\real^{n}$ are convex compact subsets; 

\noindent v) the distributional map $D:\zDomain\to (\mathcal{P}(M),W_{1})$ is continuous. 

\noindent Then the fixed point set $\text{Fix}(H)$ is nonempty and compact.  
\end{theorem}

\begin{proof}
The proof amounts to showing that $H$ satisfies the hypotheses of Kakutani's Fixed Point Theorem \cite[Corollary 17.55]{guide2006infinite} for correspondences (set-valued functions). Since the domain $\Domain$ is convex and compact by hypothesis, we show that $H$ has a closed graph and non-empty convex and compact set values in $P(\Domain)$. Following the Closed Graph Theorem \cite[Theorem 17.11]{guide2006infinite}, compactness of $\Domain$ implies that $H$ has closed graph if and only if it is closed valued and upper hemicontinous. Hence our proof reduces to showing that (i) $H$ has non-empty closed values, (ii) $H$ is upper hemicontinuous, and (iii) $H$ has convex values.  

Define the intermediate  functions
\begin{equation}
f(x';z) = \max_{y'\in\mathcal{Y}} \Phi(x',y';z) \quad \text{and} \quad g(y';z)=\min_{x\in\mathcal{X}} \Phi(x',y';z)
\end{equation}
as well as the realization functions
\begin{equation}
    F(z)=\arg\min_{x'\in\mathcal{X}} f(x';z) \quad \text{and} \quad G(z) = \arg\max_{y'\in\mathcal{Y}} g(y';z).
\end{equation}
 for all $x'\in\mathcal{X}$, $y'\in\mathcal{Y}$, and $z\in\Domain$. Using this convention, $H$ can be written compactly as  $H(z) = \left(F(z),G(z)\right)$. It follows from continuity of $\phi$ and $D$ on $\Domain$, as well as compactness of $\mathcal{X}$ and $\mathcal{Y}$ that $f$ and $g$ are continuous \cite[Theorem 17.31]{guide2006infinite}. The Maximum Theorem applied to $F$ and $G$ implies that $F$ and $G$ are upper hemicontinuous and have nonempty compact set values. Here, compactness implies closed-ness. Thus the values of $H$ are closed since the Cartesian product of closed sets is closed. This proves (i). 

To see that $H$ is upper hemicontinuous, fix $z\in\Domain$ and let $U$ be an open set such that $H(z)\subset U$. Then $H$ will be upper hemicontinuous provided that we can show that there exists an open neighborhood $W$ of $z$ such that $H(W)\subset U$. Given that $H(z)$ is a compact subset of $U$, \cite[Theorem 2.62]{guide2006infinite} guarantees the existence of open sets  $V_{x}\subset\mathcal{X}$ and $V_{y}\subset\mathcal{Y}$ such that $H(z)\subset V_{x}\times V_{y}\subset U$. Since $F$ and $G$ are upper hemicontinuous, then the upper inverse sets $F^{u}(V_{x}) = \{z:F(z)\subset V_{x}\}$ and $G^{u}(V_{y}) = \{z:G(z)\subset V_{y}\}$ are open in $\Domain$. Let $W=F^{u}(V_{x})\cap G^{u}(V_{y})$. Then $z\in W$ by construction, so $W,H(W)\neq\emptyset$. Furthermore, $W$ is an open neighborhood of $z$ and $H(W)\subset V_{x}\times V_{y}\subset U$. Thus condition (ii) holds. 

Observe that since $x'\mapsto f(x';z)$ is convex for all $z$ and $\mathcal{X}$ is convex, then $F(z)$ is convex for all $z\in\Domain$. Similarly, $G(z)$ is convex for all $z$. Since the Cartesian product of convex sets is convex, then condition (iii) follows. 
\end{proof}

Recall that the intuition for the equilibrium points is that they are the saddle points of the stationary saddle point poroblem that they induce. In this next results, we summarize this characterization. 

\begin{proposition}{(\textit{\textbf{Saddle point and Equilibrium Equivalence}})}





\noindent Suppose that an equilibrium point exists. Then  $(\bar{x},\bar{y})\in\Domain$ is an equilibrium point if and only if 
\begin{equation}
\label{eqn:eqSaddle}
    \Phi(\bar{x},y;\bar{x},\bar{y}) \leq \Phi(\bar{x},\bar{y};\bar{x},\bar{y}) \leq \Phi(x,\bar{y};\bar{x},\bar{y})
\end{equation}
for all $x\in\mathcal{X}$ and $y\in\mathcal{Y}$. 
\end{proposition}

We omit the proof as it is amounts to the same proof technique for the classical saddle point characterization result.

We will leverage the results of this section in the analysis of first-order methods that will be utilized to solve the stochastic minmax problem. In the following, we outline some working assumptions used in the algorithmic synthesis and analysis, and provide additional intermediate results.

\subsection{Equilibrium Points for Strongly Monotone Gradient Maps}

In what follows, we outline relevant assumptions that we use in this paper for the synthesis and analysis of first-order deterministic and stochastic algorithms to identify equilibrium points. 

\begin{assumption}{(\textit{\textbf{Strong-Convexity-Strong-Concavity}})}\label{as:scsc}
The function $\phi$ is is continuously differentiable over $\realDomain$ for any realization of $w$. The function $\phi$ $\gamma$-strongly-convex-strongly-concave, for any realization of $w$; that is, $\phi$ is $\gamma$-strongly-convex in $x$ for all $y\in\real^{n}$ and $\gamma$-strongly-concave in $y$ for all $x\in\real^{d}$. 
\end{assumption}

\begin{assumption}{(\textit{\textbf{Joint Smoothness}})}\label{as:jointSmooth}
The stochastic gradient map $\psi$ given by\\ $\psi(z,w) := (\nabla_{x}\phi(z,w),-\nabla_{y}\phi(z,w))$ is $L$-Lipschitz in $z$ and $w$. Namely, 
\begin{equation*}
    \norm{ \psi(z,w) - \psi(z',w)  }  \leq L \norm{z-z'}, \quad
    \norm{ \psi(z,w) - \psi(z,w')  }  \leq L \dist(w,w').
\end{equation*}
for any $z,z'\in\realDomain$ and $w,w'$ supported on $M$. Here $\dist:M\times M \rightarrow \real$ denotes the metric on $M$. 
\end{assumption}

\begin{assumption}{(\textit{\textbf{Distributional Sensitivity}})} \label{as:sensitive}
The distributional map \\ $D:\real^{d}\times\real^{n} \rightarrow\mathcal{P}(M)$ is $\varepsilon$-Lipschitz. Namely, 
\begin{equation*}
    W_{1}(D(z),D(z'))\leq \varepsilon\norm{z-z'}
\end{equation*}
for any $z,z'\in\realDomain$, where $W_{1}$ is the Wasserstein-1 distance.  
\end{assumption}

\begin{assumption}{(\textit{\textbf{Compact Convex Sets}})} 
\label{as:sets}
The sets $\mathcal{X}\subset\real^{n}$ and $\mathcal{Y}\subset\real^{m}$ are compact and convex. 
\end{assumption}

Typically, the assumption of strong-convexity-strong-concavity enables linear convergence to saddle-points in standard primal-dual methods~\cite{koshal2011multiuser}.  Furthermore, strong-convexity-strong-concavity implies uniqueness of saddle point solutions; this allows to derive convergence results to the unique saddle-point in the static case, and tracking results in the context of time-varying minmax problems~\cite{dall2020optimization}. 
We also note that this assumption is useful in this paper in order to characterize the intrinsic relationship between optimal solutions to \cref{problemStatement} and equilibrium points. We also note that, for simplicity, the assumption imposes a common geometry parameter in $\phi$ for both the $x$ and $y$ values; however, our analysis is the same for functions $\phi$ being $\gamma_{1}$-strongly-convex in $x$ and $\gamma_{2}$-strongly-concave in $y$ (as we can take $\gamma=\min\{\gamma_{1},\gamma_{2}\}$).

Assuming that the distributional map is $\varepsilon$-Lipschitz and the gradient is Lipschitz in the random variable is commonplace in the literature on decision-dependent distributions to characterize the overall effects of the distributional maps on the random variables~\cite{drusvyatskiy2022stochastic,perdomo2020performative,wood2021online}. Since we assume the support of our random variables $w$ reside in a complete and separable metric space (Polish space), then a natural way to relate the resulting distributions is via the Wasserstein-1 metric. Following Kantorovich-Rubenstein Duality \cite{bogachev2012monge,kantorovich1958space}, this metric can be written as
\begin{equation*}
    W_{1}(\mu,\nu) = \sup\bigg\{\underset{w\sim \mu}{\EX}[g(w)]-\underset{w\sim \nu}{\EX}[g(w)] \ 
    \bigg\vert \ g:M\rightarrow\real, \ \text{Lip}(g)\leq 1\bigg\}
\end{equation*}
for all $\mu,\nu\in\mathcal{P}(M)$. Here the supremum is taken over all Lipschitz-continuous functionals on $M$ with Lipschitz constant less than or equal to one.

Closed and convex constraint sets are common in the literature on primal-dual methods, which are the main algorithms that will be considered shortly~\cite{hale2017asynchronous,koshal2011multiuser}. Due to Heine-Borel, compactness of $\mathcal{X}$ and $\mathcal{Y}$ simply means closed and bounded. The addition of boundedness here is not restrictive; one can assume boundedness while the underlying sets can still be made arbitrarily large to include the saddle-points. As an illustration, consider the closed rectangles $\mathcal{X}= [-r,r]^{n}$ and $\mathcal{Y}= [-r,r]^{m}$ for some $r>0$. Then $\mathcal{X}$ and $\mathcal{Y}$ are compact and convex for any $r>0$, and $r$ can be made an arbitrarily large positive number. See, e.g.,~\cite{koshal2011multiuser} for an example in the context of constrained optimization problems. 

To proceed, we cast the equilibrium point problem into the variational inequality framework. We show that the equilibrium problem is equivalent to a variational inequality over $\mathcal{Z}:=\mathcal{X}\times\mathcal{Y}$, where we use the stacked variable $z=(x,y)$ when convenient. We then demonstrate uniqueness of the equilibrium points for saddle point problems that satisfy the above assumptions.

Recall that in \cref{as:jointSmooth}, we introduce the stochastic gradient map $\psi$ defined by $\psi(z,w) =( \nabla_{x}\phi(z,w),-\nabla_{y}\phi(z,w))$.  Using this convention, we denote the decoupled gradient map as
$$\Psi(z;z') = \underset{w\sim D(z')}{\EX}[\psi(z,w)] . $$

This motives the following characterization, which highlights the fact that equilibrium points are solutions to the decoupled gradient variational inequality.
\begin{theorem}{(\textit{\textbf{Equilibrium Variational Inequality}})}
A point $\bar{z}\in\mathcal{Z}$ is an equilibrium point provided that 
\begin{equation}
\label{equilibrium_variational_inequality}
    \langle z - \bar{z}, \Psi(\bar{z};\bar{z}) \rangle \geq 0 \, , \quad \forall z\in\mathcal{Z}.
\end{equation}

\end{theorem}
Proof of this fact follows steps that are similar to the ones in \cite[Example 12.50]{rockafellar2009variational}. In light of \cref{def:equilibrium}, this result suggest that $\bar{z}$ are solutions to variational inequality induced by the stationary distribution $D(\bar{z})$.
In the following, we show that when $z\mapsto\Psi(z;z')$ is strongly monotone for all $z'\in\realDomain$, a unique equilibrium point exists. Furthermore, under this assumption, we can show that the distance between the saddle points for the original problem in \cref{problemStatement} and the unique equilibrium point is bounded.

\begin{proposition}
\label{cor:strongMonotone}
Suppose that \cref{as:scsc} holds. Then, for any $w\in M$, $z\mapsto\psi(z,w)$ is $\gamma$-strongly-monotone. Furthermore, for any $z'\in\realDomain$, $z\mapsto\Psi(z;z')$ is $\gamma$-strongly-monotone.
\end{proposition}
Proof of this result is immediate. Below we provide a Lemma that allows us to characterize the changes in the distributional argument of the decoupled gradient map $\Psi$. This amounts to the decoupled gradient map being Lipschitz continuous in the distributional argument. 

\begin{lemma}{(\textit{\textbf{Gradient Deviations}})}
\label{lem:gradient_deviations}
Suppose that \crefrange{as:jointSmooth}{as:sets} hold. Then, for any $\hat{z}\in\realDomain$, the map $z\mapsto\Psi(\hat{z},z)$ is $\varepsilon L$-Lipschitz. Furthermore, the restriction to $\zDomain$ is bounded in the following way: for any $\hat{z}\in\zDomain$
\begin{equation}
\label{lem:diamBound}
    \norm{\Psi(\hat{z};z)-\Psi(\hat{z};z')}\leq \varepsilon L D_{\mathcal{Z}}
\end{equation}
for all $z,z'\in\zDomain$ where $D_{\mathcal{Z}}= \text{diam}(\zDomain)<\infty$.
\end{lemma}

\begin{proof}
Let $v\in\realDomain$ be an arbitrary unit vector and fix $\hat{z},z,z'\in\realDomain$. It follows that
\begin{equation*}
   \langle  v , \Psi(\hat{z},z)-\Psi(\hat{z};z') \rangle = \underset{w\sim D(z)}{\EX}[\langle v, \psi(\hat{z},w)\rangle ] - \underset{w\sim D(z')}{\EX}[\langle v, \psi(\hat{z},w)\rangle ]. 
\end{equation*}
By our assumption, we have that $w\mapsto \langle v,\psi(\hat{z},w)\rangle $ is Lipchitz with constant $L\norm{v}=L$. Thus, from Kantorivich and Rubenstein, we have that 
\begin{equation*}
    \underset{w\sim D(z)}{\EX}[\langle v , \psi(\hat{z},w)\rangle ] - \underset{w\sim D(z')}{\EX}[\langle v,\psi(\hat{z},w)\rangle ]\leq LW_{1}(D(z),D(z'))\leq \varepsilon L \norm{z-z'},
\end{equation*}
where that last inequality follows from $\varepsilon$-sensitivity of $D$. Thus we have that for any unit vector $v$, $\langle v, (\Psi(\hat{z},z)-\Psi(\hat{z};z')\rangle  \leq \varepsilon L \norm{z-z'}.$
Hence, choosing
    $v = (\Psi(\hat{z},z)-\Psi(\hat{z};z'))/\norm{(\Psi(\hat{z},z)-\Psi(\hat{z};z'))}$
yields the result. Lastly, by \cref{as:sets}, $\zDomain$ is compact and hence $\norm{z-z'}\leq D_{\mathcal{Z}}<\infty$ for any $z,z'\in\zDomain$. Thus, \cref{lem:diamBound} follows. 
\end{proof}

In what follows, we demonstrate existence and uniqueness of equilibrium points. Similar to the statement of existence, we show that $H$ satisfies the Banach-Picard Fixed Point Theorem by providing conditions for which $H$ is a strict contraction.

\begin{theorem}{(\textit{\textbf{Existence and Uniqueness of Equilibrium Points}})}
Suppose that \crefrange{as:scsc}{as:sets} hold. Then:
\begin{enumerate}
    \item For all $z,z'\in\zDomain$, $\norm{H(z)-H(z')}\leq \frac{\varepsilon L}{\gamma}\norm{z-z'}$,
    \item If $\frac{\varepsilon L}{\gamma}<1$, then there exists a unique equilibrium point $(\bar{x},\bar{y})\in\Domain$.
\end{enumerate}
\end{theorem}
\begin{proof}
Let $\hat{z},\tilde{z}\in\zDomain$ be fixed. Then the maps $z\rightarrow\Psi(z;\hat{z})$ and $z\rightarrow\Psi(z;\tilde{z})$ are $\gamma$-strongly-strongly monotone. Furthermore, our strong-convexity and strong-concavity assumptions on $\phi$ imply that $H(\hat{z})$ and $H(\tilde{z})$ and are single valued in $\Domain$. Recall from our definition of $H$ that $H(\hat{z})$ and $H(\tilde{z})$ are solutions to the variational inequalities induced by $\hat{z}$ and $\tilde{z}$ respectively. That is, for all $z\in\Domain$, 
\begin{equation}
\label{thm:variational_solution}
    \langle z - H(\hat{z}), \Psi(H(\hat{z});\hat{z}) \rangle \geq 0 \quad \text{and} \quad  \langle z - H(\tilde{z}), \Psi(H(\tilde{z});\tilde{z}) \rangle \geq 0.
\end{equation}
It follows from strong monotonicity that  $\langle H(\hat{z}) - H(\tilde{z}), \Psi(H(\hat{z});\hat{z})-\Psi(H(\tilde{z});\hat{z}) \rangle \geq \gamma\norm{H(\hat{z}) - H(\tilde{z})}^{2}$, and  \cref{thm:variational_solution} imply that $\langle H(\tilde{z}) -  H(\hat{z}) ,\Psi(H(\hat{z});\hat{z}\rangle\geq 0$. Hence, 
\begin{equation}
\label{prop:unique_one}
\langle H(\hat{z}) - H(\tilde{z})),\Psi(H(\tilde{z});\hat{z}\rangle \leq -\gamma\norm{H(\hat{z}) - H(\tilde{z})}^{2}.
\end{equation}
To proceed, we provide a lower bound for the quantity on the left-hand side. By applying Cauchy-Schwarz and \cref{lem:gradient_deviations}, we get that 
\begin{equation*}
    \langle H(\hat{z})-H(\tilde{z}), \Psi(H(\tilde{z});\tilde{z}) -  \Psi(H(\tilde{z});\hat{z})\rangle
    \leq \varepsilon  L \norm{H(\hat{z})-H(\tilde{z})}\norm{\tilde{z}-\hat{z}}. 
\end{equation*}
Since \cref{thm:variational_solution} implies that $\langle H(\hat{z})-H(\tilde{z}), \Psi(H(\tilde{z});\tilde{z}\rangle \geq 0$, then we get that 
\begin{equation}
\label{prop:unique_two}
  \langle H(\hat{z})-H(\tilde{z}) , \Psi(H(\tilde{z});\hat{z} \rangle 
  \geq -\varepsilon L \norm{H(\hat{z})-H(\tilde{z})}\norm{\tilde{z}-\hat{z}}. 
\end{equation}
Combining inequalities \cref{prop:unique_one} and \cref{prop:unique_two} yields
\begin{equation*}
    -\gamma\norm{H(\hat{z}) - H(\tilde{z})}^{2}
    \geq -\varepsilon L \norm{H(\hat{z})-H(\tilde{z})}\norm{\tilde{z}-\hat{z}},
\end{equation*}
and simplifying yields the result.

Since $H$ is Lipschitz continuous, it is a strict contraction if $\varepsilon L/\gamma < 1$. Uniqueness of the fixed point follows from the Banach-Picard Fixed-Point Theorem. 
\end{proof}

We have demonstrated existence and uniqueness of equilibrium points for some classes of problems; next, we characterize the relationship between equilibrium points and solutions of the original problem in~\cref{problemStatement}. First, an important observation is that when $\varepsilon=0$, the problem statement in \cref{tiProblem} has a stationary probability distribution with respect to the decisions. Hence, saddle points coincide with equilibrium points. When $\varepsilon > 0$, we  provide a guarantee on the distance between solutions of the two problems.

\begin{proposition}{\textit{\textbf{(Bounded Distance)}}}
\label{prop:distance}
Suppose that \crefrange{as:scsc}{as:sets} hold. Let $z^*$ be the optimal solution of~\cref{problemStatement}, and let $\bar{z}$ be the equilibrium point. Then, 
\begin{equation}
    \norm{z^{*}-\bar{z}}\leq \frac{\varepsilon L}{\gamma}D_{\mathcal{Z}}. 
\end{equation}

\end{proposition}

\begin{proof}
From the optimality conditions, we have that the decoupled gradient map satisfies $\langle \bar{z} - z^* , \Psi(z^*;z^*) \rangle \geq 0$ and $\langle z^* - \bar{z} , \Psi(\bar{z};\bar{z}) \rangle \geq 0$.
By combining these results with results with our gradient deviation bound in \cref{lem:gradient_deviations}, we obtain the following: 
\begin{align*}
    \langle \bar{z} - z^* , \Psi(\bar{z};\bar{z}) - \Psi(z^*;\bar{z}) \rangle
    &=  \langle \bar{z} - z^* , \Psi(\bar{z};\bar{z}) \rangle -  \langle \bar{z} - z^*,\Psi(z^*;\bar{z}) \rangle  \\
    &\leq  \langle \bar{z} - z^* , \Psi(z^*;z^*) - \Psi(z^*;\bar{z}) \rangle \\
    &\leq \norm{\bar{z} - z^*} \ \norm{\Psi(z^*;z^*) - \Psi(z^*;\bar{z})} \\
    & \leq \varepsilon L D_{\mathcal{Z}}\norm{\bar{z}-z^{*}} , 
\end{align*}
where the second to last step follows from the Cauchy-Schwarz inequality. It follows from $\gamma$-strong-monotonicity that 
\begin{equation*}
    \gamma\norm{\bar{z}-z^*}^2 \leq \langle \bar{z} - z^* , \Psi(\bar{z};\bar{z}) - \Psi(z^*;\bar{z}) \rangle  \leq \varepsilon L D_{\mathcal{Z}}\norm{\bar{z}-z^{*}}
\end{equation*}
so that canceling terms and dividing by $\gamma$ yields the result. 
\end{proof}


\subsection{Finding the Equilibrium Point via Primal-Dual Algorithm}
In this section, we focus on a primal-dual method for finding the equilibrium point, and we provide results in terms of convergence at a linear rate. 
In particular, we focus on the Equilibrium Primal-Dual (EPD) algorithm, which is based on the algorithmic map: 
\begin{equation}
    \G(z;z') := \Pi_{\zDomain}\left(z - \eta\Psi(z;z')\right)
\end{equation}
for all $z,z'\in\realDomain$, where we recall that  $\Pi_{\zDomain}(z)=\arg\min_{z'\in\zDomain} \norm{z-z'}^{2}$ is the projection map and $\eta>0$ is a positive step size. Given an initial point $z_{0}\in\realDomain$, the algorithm then generates a sequence via the Banach-Picard iteration: 
\begin{equation}
z_{t+1} = G(z_{t};z_{t}) = \Pi_{\zDomain}\left(z_{t} - \eta\Psi(z_{t};z_{t})\right), \,\,\, t = 1, 2, \ldots
\end{equation}
where we recall that $\Psi(z_{t};z_{t}) = \EX_{w\sim D(z_{t})}[\psi(z_{t},w)]$, with $D(z_{t})$ the distribution induced by $z_{t}$. A key feature of this method is that each step projects onto the constraint sets, and hence $z_{t}\in\zDomain$ for all $t\geq 1$ for any initial condition $z_{0}$. In the following we demonstrate that the equilibrium points are fixed points of the algorithmic map. We then provide linear convergence results for the case when our assumptions hold on just the constraint sets $\zDomain$, and later globally.

\begin{proposition}{\textit{\textbf{(Fixed Point Characterization)}}}
\label{prop:equilibrium_fixed_point}
Let \crefrange{as:scsc}{as:sets} hold and suppose that $\frac{\varepsilon L }{\gamma}<1$. A point $\bar{z}\in\zDomain$ is an equilibrium point if and only if $\bar{z}=\mathcal{G}(\bar{z};\bar{z})$.
\end{proposition}

\begin{proof}
We want to show that $\bar{z}$ solving the variational inequality in \cref{equilibrium_variational_inequality} is equivalent to $\bar{z}=\Pi_{\zDomain}(\bar{z}-\eta\Psi(\bar{z};\bar{z}))$ for $\eta>0$.  From \cite[Theorem 1.5.5]{facchinei2007finite}, for any $\hat{z}\in\realDomain$, $\Pi_{\zDomain}(\hat{z})$ is the unique element of $\zDomain$ such that
\begin{equation}
\label{prop:projIneq}
    \langle z - \Pi_{\zDomain}(\hat{z}), \Pi_{\zDomain}(\hat{z}) - \hat{z} \rangle \geq 0 
\end{equation}
holds for any $z\in\zDomain$. 
As for the forward direction, if $\bar{z}$ as equilibrium point, then \cref{equilibrium_variational_inequality} is equivalent to 
\begin{equation}
\label{vi_equivalent}
    \langle z - \bar{z}, \bar{z} - (\bar{z} - \eta\Psi(\bar{z};\bar{z}) ) \rangle \geq 0.
\end{equation}
In setting $\hat{z} = \bar{z}-\eta\Psi(\bar{z};\bar{z})$ in \cref{prop:projIneq}, we get that $\bar{z}=\Pi_{\zDomain}(\bar{z}-\eta\Psi(\bar{z};\bar{z}))$. Conversely, if $\bar{z}$ is such that $\bar{z}=\Pi_{\zDomain}(\bar{z}-\eta\Psi(\bar{z};\bar{z}))$, then by substituting $\hat{z} = \bar{z}-\eta\Psi(\bar{z})$ into \cref{vi_equivalent}, we have that $\bar{z}$ satisfies \cref{equilibrium_variational_inequality}.
\end{proof}

Now that we can equivalently characterize equilibrium points of the EPD map, we can leverage fixed point theory to analyze the EPD algorithm. In the following, we demonstrate convergence to the unique equilibrium point.

\begin{theorem}{(\textbf{\textit{EPD Convergence}})}
\label{thm:epd_convergence}
\label{thm:equilibrium_convergence} Suppose that \crefrange{as:scsc}{as:sets} hold and that $\frac{\varepsilon L }{\gamma}<1$. Then the sequence $z_{t+1} = \mathcal{G}(z_{t};z_{t})$ satisfies the bound 
\begin{equation}
\label{thm:equilibrium_contraction}
    \norm{z_{t} - \bar{z}} \leq \alpha^{t} \norm{z_{0}-\bar{z}}
\end{equation}
for any initial point $z_{0}\in\zDomain$, and for $\alpha :=\sqrt{1-2\eta\gamma+\eta^{2}L^{2}}+\eta\varepsilon L$. Furthermore, if 
\begin{equation}
    \label{thm:convergence_stepsize_condition}
    \eta \in \left(0, \frac{2(\gamma - \varepsilon L)}{L^{2}(1-\varepsilon^{2})}\right),
\end{equation}
then $z_{t}$ converges linearly to the unique equilibrium point $\bar{z}$.
\end{theorem}

\begin{proof}
By applying our fixed point result in \cref{prop:equilibrium_fixed_point} and the triangle inequality, we get that
\begin{equation}
    \norm{z_{t+1}-\bar{z}} = \norm{\G(z_{t};z_{t}) - \G(\bar{z};\bar{z})} \leq \norm{\G(z_{t};z_{t}) - \G(z_{t};\bar{z})} + \norm{\G(z_{t};\bar{z}) - \G(\bar{z};\bar{z})}.
\end{equation}
Bounding the first quantity amounts to applying our gradient deviation result in \cref{lem:gradient_deviations}. Hence, $\norm{\G(z_{t};z_{t}) - \G(z_{t};\bar{z})} \leq \eta \norm{\Psi(z_{t};z_{t})-\Psi(z_{t};\bar{z})} \leq \eta\varepsilon L\norm{z_{t}-\bar{z}}$.

The second quantity is the standard analysis for stationary primal-dual. Namely, 
\begin{align*}
    \norm{\G(z_{t};\bar{z})-\G(\bar{z}';\bar{z})}^{2} 
    & \leq \norm{(z_{t}-\bar{z})-\eta(\Psi(z_{t};\bar{z})-\Psi(\bar{z};\bar{z})}^{2} \\
    & = \norm{z_{t}-\bar{z}}^{2} + \eta^{2}\norm{(\Psi(z_{t};\bar{z})-\Psi(\bar{z};\bar{z})}^{2} \\
    &\qquad \qquad \ \quad  - 2\eta\langle z_{t} - \bar{z}, \Psi(z_{t};\bar{z})-\Psi(\bar{z};\bar{z})\rangle \\ 
    & \leq (1-2\eta\gamma + \eta^{2}L^{2})\norm{z_{t}-\bar{z}}^{2}.
\end{align*}
hence adding yields $\norm{z_{t}-\bar{z}} \leq (1-2\eta\gamma + \eta^{2}L^{2})\norm{z_{t-1}-\bar{z}}$ so that repeated application yields the result in \eqref{thm:equilibrium_contraction}. Convergence requires choosing step-size $\eta>0$ such that $0<\alpha<1$. We observe that if $0<\eta<2\gamma / L^{2}$, then the quantity $\sqrt{1-2\eta\gamma + \eta^{2}L^{2}}$ is real-valued. 

Additionally, we find that $\alpha < 1$ provided that $0 < \eta \left( \eta L^{2}(\varepsilon^{2} -1) - 2(\varepsilon L -\gamma) \right)$ and hence we must have that  $\eta L^{2}(\varepsilon^{2} -1) - 2(\varepsilon L -\gamma) > 0$. Finally, we note that 
    \begin{equation*}
    \frac{2(\gamma-\varepsilon L)}{L^{2}(1-\varepsilon^{2})} \leq \frac{2\gamma}{L^{2}},
    \end{equation*}
    thus the result follows. 
\end{proof}

In the next section, we focus on a stochastic algorithm. This stochastic method operates as an inexact version of the primal-dual algorithm, a fact which we highlight in our results. 

\subsection{Stochastic Projected Primal-Dual Method}
In the previous section, we showed convergence of a primal-dual algorithm. However, this algorithm requires the computation of the gradient;  that is, one can either compute the function $\EX_{w\sim D(z')} [\psi(z,w)]$ in closed form, or we have access to a black box evaluator that can compute this exactly for for any $z'\in\realDomain$. In the first case, this means that we have explicit knowledge of the distributional map $D$. The latter assumes that the integral expression can be computed exactly or approximated to machine precision. However, this integral approximation problem is known to suffer from the curse of dimensionality. 

When the computation of the gradient is not possible, a common approach is to utilize a stochastic gradient estimate.  Accordingly, in this work we assume access to an oracle  $\Omega:\realDomain\rightarrow\realDomain$ that returns a gradient estimator at a given point in the domain. Conceptually, $\Omega$ is  a stochastic process indexed over $\zDomain$ and may take the form of common estimators used in the stochastic optimization literature:
\[  \Omega(z) = 
    \begin{cases} 
    \psi(z,w_{1}), \quad w_{1}\sim D(z) \\
    \frac{1}{N}\sum_{i=1}^{N} \psi(z,w_{i}), \quad w_{1},\hdots,w_{N}\overset{i.i.d.}{\sim} D(z) \\
    \end{cases}
\]
where $\{w_{i}\}_{i=1}^{N}$ are random samples drawn from $D(z)$. The first two represent the stochastic gradient and mini-batch gradient estimator, respectively. The third one is a noisy gradient evaluation that models noisy black-box evaluations of the true gradient map $\Psi$. Assuming access to $\Omega$, we define the Stochastic Equilibrium Primal-Dual (SEPD) method via the algorithmic map:
\begin{equation*}
\hat{\G}(z) = \Pi_{\zDomain}\left(z - \eta\Omega(z) \right)
    \end{equation*}
so that the SEPD algorithm generates the sequence $z_{t+1}=\hat{\G}\left(z_{t}\right) , \,\,\,\, t \geq 0$
starting from a point $z_0 \in \realDomain$. In our analysis, we associate the gradient error with a class of potentially heavy tailed probability distributions. We consider the class of sub-Weibull random variables as defined next. 
\begin{definition}{(\textit{\textbf{Sub-Weibull Random Variable}}~\cite{vladimirova2020sub})}
The distribution of a random variable $\xi$ is sub-Weibull, denoted $\xi\sim\sW(\theta,\nu)$, if there exists $\theta > 0,\nu>0$ such that 
$\norm{z}_{k}\leq \nu_{2} k^{\theta}$, 
for all $k\geq1$.
\end{definition}

In this definition, $\theta$ measures the heaviness of the tail (higher values of $\theta$  correspond to 
heavier tails); for example, $\theta=1$ and $\theta=1/2$ correspond to sub-exponential and sub-Gaussian random variables respectively. The parameter $\nu$ represents a proxy for the variance of $\xi$ \cite{vladimirova2020sub,wong2020lasso}. 

In addition to the fact that heavy tailed distributions are well-motivated in the literature on stochastic gradient methods~\cite{simsekli2019tail,gurbuzbalaban2020heavy}, they allow us to derive results that are applicable to also sub-exponential and sub-Gaussian random variables, as well as random variables whose distribution has a finite support. 

In the following we provide results for the first moment and second moment. These results will require some of the next assumptions.  

\begin{assumption}{(\textit{\textbf{Stochastic Framework}})}
\label{as:stochastic_framework}
Let $\Omega$ be a gradient estimator with random variables $w$ defined over the probability space $(\mathcal{S},\mathcal{F},\mathbf{P})$ where $\mathcal{S}$ is a sample space with $\sigma$-algebra $\mathcal{F}$, and probability measure $\mathbf{P}$. Suppose that the following hold.
\begin{enumerate}
\item \textit{\textbf{Filtration}}. Let $(\mathcal{S}, \mathcal{F}, \mathbf{F},\mathbf{P})$ be a filtered probability space with  filtration $\mathbf{F}=\{\mathcal{F}_{t}\}_{t\geq0}$ where $\mathcal{F}_{0}=\{\emptyset,\Omega\}$ and $\Omega$ is $\mathcal{F}_{t}$ measurable for all $t\geq0$.  
\item \textit{\textbf{Sub-Weibull Error}}. Let $\Omega$ be such that there exist tail parameter $\theta>0$ and bounded variance proxy function $\nu:\real^{d}\times\real^{n}\mapsto \real_{\geq0}$ such that
\begin{equation}
    \xi(z) = \norm{\Omega(z) - \underset{w\sim D(z)}{\EX}[\psi(z,w)]}\sim \sW(\theta,\nu(z))
\end{equation}
for all $z\in\zDomain$. Denote $\tilde{\nu}\in(0,\infty)$ such that $\nu(z)\leq \tilde{\nu}$ for all $z\in\zDomain$.  
\item \textit{\textbf{Unbiased Gradient Estimator}}. The gradient estimator $\Omega$ is unbiased in that $\EX_{w\sim D(z)} \Omega(z) = \Psi(z;z)$ for all $z\in\zDomain$.
\end{enumerate}
\end{assumption}

While we assume that the gradient estimator may be defined over the whole space $\realDomain$, notice that we only place the sub-Weibull assumption on the restriction of $\Omega$ to  the constraint set $\zDomain$. A typical assumption in the literature is that the norm of $\xi$ is uniformly bounded in expectation. Here, we assume that the norm of the gradient error is distributed according to a heavy-tailed distribution. If the variance proxy function $\nu$ is continuous over the compact set $\zDomain$, we retrieve the uniform boundedness property. By assuming $\theta$ is fixed for any $z\in\zDomain$, we assume that all realizations of the process belong to the same sub-Weibull class.

\begin{theorem}{(\textit{\textbf{First Moment}})}
\label{thm:stochErrorBound} 
Suppose that \crefrange{as:scsc}{as:sets} hold and $\frac{\varepsilon L}{\gamma}<1$. If $\Omega$ satisfies \cref{as:stochastic_framework}.2, and $\eta$ satisfies the bound in \cref{thm:convergence_stepsize_condition} then the following hold:
\begin{enumerate}
    \item \textit{\textbf{Expectation}}. The sequence $\{z_{t}\}_{t\geq0} $ satisfies the bound in expectation
    \begin{equation}
    \label{thm:first_moment_expectation}
        \EX\norm{z_{t}-\bar{z}} \leq \alpha^{t}\norm{z_{0}-\bar{z}} + \frac{\bar{\nu}\eta}{1-\alpha}.
    \end{equation}
    for all $t\geq0$. 
    \item \textit{\textbf{High Probability}}. For any $\delta\in(0,1)$, and $t\geq0$,
    \begin{equation}
    \label{thm:first_moment_high_probability}
    \mathbb{P}\left(\norm{z_{t}-\bar{z}} \leq \alpha^{t}\norm{z_{0}-\bar{z}} + c(\theta)\log^{\theta}\left(\frac{2}{\delta}\right) \frac{\bar{\nu}\eta}{1-\alpha}\right) \geq 1-\delta
    \end{equation}
 with $c(\theta) := \left(\frac{2e}{\theta}\right)^\theta$.
\end{enumerate}
\end{theorem}

Before proving the result of the theorem, we provide the following supporting lemmas that will be utilized in the proof. 

\begin{lemma}{(\textit{\textbf{Equivalent Characterizations}})} If $\xi$ is a sub-Weibull random variable with tail parameter $\theta > 0$, then the following characterizations are equivalent (we recall that $\norm{z}_{k} = \EX[\vert z\vert^{k}]^{1/k}$):
\begin{enumerate}
    \item[(c1)] Tail Probability: $\exists \ \nu_{1}>0$ such that $\mathbb{P}(|z| \geq \epsilon ) \leq 2\exp(-\left(\epsilon/\nu_{1})\right)^{1/\theta}$ 
    for all $\epsilon>0$.
    \item[(c2)] Moment: $\exists \ \nu_{2}>0$ such that $\norm{z}_{k}\leq \nu_{2} k^{\theta} $ for all $k\geq1$.
\end{enumerate}
Moreover, if (c2) holds for a given $\nu_{2} > 0$, then (c1) holds with $\nu_{1} = \left(\frac{2e}{\theta}\right)^{\theta}\nu_{2}$. 
\end{lemma}

\begin{lemma}{(\textbf{\textit{Sub-Weibull Inclusion}})} 
\label{prop:inclusionSubWeibull}
If $\xi\sim\sW(\theta,\nu)$ based on (c2) and $\theta',\nu'>0$ such that $\theta\leq\theta'$ and $\nu\leq\nu'$ then $\xi\sim\sW(\theta',\nu')$.
\end{lemma}

\begin{lemma}{(\textbf{\textit{Sub-Weibull Closure}})} 
\label{prop:closureSubWeibull}
If $\xi_{1} \sim \sW(\theta_1,\nu_1)$, $\xi_{2} \sim \sW(\theta_2,\nu_2)$ are (possibly coupled) sub-Weibull random variables based on (c2) and $c\in\mathbb{R}$, then the following hold:
\begin{enumerate}
    \item  $\xi_{1}+\xi_{2}\sim\sW(\max \{\theta_1,\theta_2\}, \nu_1+\nu_2)$;
    \item  $\xi_{1}\xi_{2}\sim \sW(\theta_1 + \theta_2, 
    \psi(\theta_1, \theta_2) \nu_1 \nu_2)$, $\psi(\theta_1, \theta_2) := (\theta_1 + \theta_2)^{\theta_1 + \theta_2} / (\theta_1^{\theta_1} \theta_2^{\theta_2})$;
    \item $c\xi_{1}\sim \sW(\theta_1, |c|\nu_1)$. 
\end{enumerate}
\end{lemma}
The proofs of these lemmas can be found in \cite{vladimirova2020sub,wong2020lasso}.

\begin{proof} \emph{of Theorem~\cref{thm:stochErrorBound}}.
For notational convenience, we denote the equilibrium error as $e_{t}=\norm{z_{t}-\bar{z}}$ throughout. We proceed by treating the stochastic gradient step as an inexact step using the triangle inequality: 
\begin{equation}
\label{proof:inexact}
e_{t+1} = \norm{\hat{\G}(z_{t}) - \G(\bar{z};\bar{z})} \leq \norm{\hat{\G}(z_{t}) - \G(z_{t};z_{t})} + \norm{\G(z_{t};z_{t})-\G(\bar{z};\bar{z})}.
\end{equation}
To bound the first quantity, Applying non-expansiveness and \cref{as:stochastic_framework}.2 yields
\begin{equation*}
\norm{\hat{\G}(z_{t}) - \G(z_{t};z_{t})} \leq \eta\norm{\Omega(z_{t}) - \Psi(z_{t};z_{t})} = \eta\xi(z_{t}). 
\end{equation*}
It follows from \cref{thm:epd_convergence} that $ \norm{\G(z_{t};z_{t})-\G(\bar{z};\bar{z})} \leq \alpha \norm{z_{t}-\bar{z}}$. 

Combining these results yields
\label{staticStochasticRecusrion}
    $e_{t+1} \leq \alpha^{t+1} e_{0} + \eta \sum_{i=0}^{t}\alpha^{i} \xi_{t-i}$.
Taking the expectation yields
    $\EX [e_{t+1}] \leq \alpha^{t+1} e_{0} + \sum_{i=0}^{t}\alpha^{i} \EX[\xi_{t-i}]$, and applying the results of Lemmas~\cref{prop:inclusionSubWeibull} and~\cref{prop:closureSubWeibull} gives us the expected result in \cref{thm:first_moment_expectation}.

Now, denote $\omega_{t+1}=\alpha^{t+1}e_{0}$ and $\sigma_{t} = \eta \sum_{i=1}^{t}\alpha^{i}\xi_{t-i}$ so that $e_{t+1}\leq \omega_{t+1} + \sigma_{t}$. From our sub-Weibull assumption, we have that $\xi_{t-i}\sim\sW(\theta,\nu_{t-i})$, where $\nu_{t-i} = \nu(z_{t-i}).$ It follows from the closure under product and addition in \cref{prop:closureSubWeibull} that $\sigma_{t} \sim \sW(\theta,\Delta)$ with $\Delta = \bar{\nu}\eta(1 - \alpha)^{-1}$. Hence, 
\begin{equation}
    \mathbb{P}\left(\sigma_{t} \geq \epsilon \right) \leq 2\exp \left(-\frac{\theta}{2e} \left( \frac{\epsilon}{\Delta} \right)^{\frac{1}{\theta}} 
    \right).
\end{equation}
By setting the right-hand side equal to $\delta$, we find that $\epsilon  = c(\theta)\log^{\theta}\left(\frac{2}{\delta}\right)\Delta$
where $c(\theta) = (\frac{2 e}{\theta})^{\theta}$. Now, observe that our stochastic recursion implies that for any $a>0$, $\mathbb{P}(\omega_{t+1} + \sigma_{t} \geq a) \geq \mathbb{P}(e_{t+1} \geq a)$.
It follows that setting $a = \omega_{t+1} + \epsilon$ yields 
\begin{equation*}
    \mathbb{P}(e_{t+1} \leq \omega_{t+1} + \epsilon ) \geq \mathbb{P}(\omega_{t+1}+\sigma_{t} \leq \omega_{t+1} + \epsilon) = \mathbb{P}(\sigma_{t}\leq \epsilon)\geq 1-\delta,
\end{equation*}
thus the result follows by substituting the expression for $\omega_{t+1}$ and $\epsilon$.
\end{proof}
 The bounds naturally translate to convergence results by considering the limit supremum. Now we demonstrate that the algorithm converges to a neighborhood of the the equilibrium in expectation and almost surely. 

\begin{theorem}(\textit{\textbf{Neighborhood Convergence}})
\label{thm:nbhdConvergence} 
Suppose that Assumptions \cref{as:scsc}-\cref{as:sets} hold, and the $\Omega$ satisfies \cref{as:stochastic_framework}.2. Assume that $\eta$ satisfies the condition~\cref{thm:convergence_stepsize_condition}. Then, the sequence of iterates $\{z_{t}\}_{t\geq0}$ converges to a neighborhood of $\bar{z}$ in expectation and almost surely. In particular,  
$$\limsup_{t\to\infty} \EX\norm{z_{t}-\bar{z}} \leq \frac{\eta\bar{\nu}}{1-\alpha}, \quad \text{and} \quad 
 \mathbb{P}\left(\limsup_{t\to\infty} \norm{z_{t}-\bar{z}} \leq \frac{\eta\bar{\nu}}{1-\alpha}\right) = 1. $$
\end{theorem}

\begin{proof}
The limit of the expectation follows immediately from above. As for almost sure convergence, we simply apply the Borel-Cantelli Lemma. As before we let $e_{t} = \norm{z_{t}-\bar{z}}$, so that the result in \cref{thm:first_moment_expectation} can be compactly written as $\EX[e_{t}]\leq  \alpha^{t}e_{0} + \eta\bar{\nu}(1-\alpha)^{-1}$. Denote $E_{t}=\max\{0, e_{t}\}$ so that $\EX[E_{t}]\leq \alpha^{t}e_{0}$.

By Markov's inequality, $P(E_{t}\leq \epsilon ) \leq \frac{\EX[E_{t}]}{\epsilon}\leq \frac{\alpha^{t+1}e_{0}}{\epsilon}$, 
for any $\epsilon>0$. Summing over $t$ yields
    $\sum_{t=0}^{\infty} P(E_{t}\geq \epsilon) \leq \frac{e_{0}}{\epsilon(1-\alpha)}<\infty$.
It follows from the Borel-Cantelli Lemma that, since the sum of tail probabilities is finite, then  $P(\limsup_{t\rightarrow\infty}E_{t}\leq \epsilon) = 1$. Since this is true for any $\epsilon>0$, then the result follows.
\end{proof}

Notice that Theorem~\cref{thm:stochErrorBound} requires only Assumption \cref{as:stochastic_framework}.2 (and it does not require the filtration in Assumption \cref{as:stochastic_framework}.1). However,  a drawback to this first-moment analysis is that it only demonstrates convergence to a neighborhood whose radius is dictated by the proxy variance, and hence the quality of the estimator. In what follows, we demonstrate that we are able to obtain stronger convergence results at the expense of requiring our estimator to be unbiased and introducing a filtration on the probability space. First however, we provide the necessary expectation bound.

\begin{theorem}{(\textit{\textbf{Second Moment Convergence}})}
\label{thm:second_moment_convergence}
Suppose that Assumptions \crefrange{as:scsc}{as:stochastic_framework} are satisfied and denote $\EX_{t} = \EX_{w\sim D(z_{t})}[ \ \cdot \vert \mathcal{F}_{t}]$. Then the following inequalities hold: 
\begin{enumerate}
    \item \textbf{\textit{One Step Bound}}. The sequence $\{z_{t}\}_{t\geq 0}$ generated by SEPD satisfies:
    \begin{equation*}
        \label{thm:stochastic_one_step_improvement}
        \EX_{t}\norm{z_{t+1}-\bar{z}}^{2} \leq \left(1 - 2(\gamma-\varepsilon L)\eta_{t} + 2(1+\varepsilon)^{2}L^{2}\eta_{t}^{2}\right)
        \norm{z_{t}-\bar{z}}^{2} + \eta_{t}^{2} \bar{\nu}^{2} 2^{1+2\theta}
    \end{equation*}
    \item \textbf{\textit{Convergence}}. If the step size is $\eta_{t} = \ell(\kappa + t)^{-1}$ where  
    \begin{equation}
        \ell > \frac{1}{2(\gamma-\varepsilon L)} \quad \text{and} \quad \kappa > \frac{(1+\varepsilon)^{2}L^{2}}{(\gamma-\varepsilon L)^{2}}
    \end{equation}
    then, the sequence $\{z_{t}\}_{t\geq 0}$ generated by SEPD satisfies:
    \begin{equation}
    \label{thm:stochastic_convergence}
    \EX \norm{z_{t}-\bar{z}}^{2} \leq \frac{\zeta}{\kappa+t}, \ \text{where} \ \zeta := \max\left\{
    \kappa \norm{z_{0}-\bar{z}}^{2}, \frac{ \ell^{2}\bar{\nu}^{2}2^{1+2\theta}}{2(\gamma-\varepsilon L)\ell-1 }
    \right\}.
    \end{equation}
\end{enumerate}
\end{theorem}

\begin{proof}
By applying the algorithmic map, and using non-expansiveness of the projection operator we obtain the following relationship:
\begin{equation*}
    \EX_{t}\norm{ z_{t+1} - \bar{z} } 
 \leq \norm{z_{t}-\bar{z}}^{2} - 2\eta_{t}\langle z_{t}-\bar{z}, \Psi(z_{t};z_{t})-\Psi(\bar{z};\bar{z})\rangle + \eta_{t}^{2}\EX_{t}\norm{\Omega(z_{t})-\Psi(\bar{z};\bar{z})}^{2}
\end{equation*}
To bound the inner product term, we use $\gamma$-strong-monotonicity and the Gradient Deviations result from Lemma 2.8: $\langle z_{t}-\bar{z}, \Psi(z_{t};z_{t})-\Psi(\bar{z};\bar{z})\rangle
    \leq (\gamma - \varepsilon L) \norm{z_{t} - \bar{z}}^{2}$. 
From the properties of the sub-Weibull random variables, we have that
    $\EX_{t}\norm{\Omega(z_{t})-\Psi(z_{t};z_{t})}^{2} 
    \leq \bar{\nu}^{2} 2^{2\theta}$
where $\bar{\nu}$ is an upper bound on the variance proxy function, and $\theta$ is a uniform tail parameter. By applying this result, as well as Young's inequality, we get that
\begin{align*}
    \EX_{t}\norm{\Omega(z_{t})-\Psi(\bar{z})}^{2}
    & = \EX_{t}\norm{\Omega(z_{t})-\Psi(z_{t};z_{t}) + \Psi(z_{t};z_{t}) - \Psi(\bar{z};\bar{z})}^{2} \\ 
    & \leq 2\EX_{t}\norm{\Omega(z_{t})-\Psi(z_{t};z_{t})}^{2} + 2\EX_{t}\norm{\Psi(z_{t};z_{t})-\Psi(\bar{z};\bar{z})}^{2} \\ 
    & \leq \bar{\nu}2^{1+\theta} + 2(1+\varepsilon)^{2}L^{2}\norm{z_{t}-\bar{z}}^{2},
\end{align*}
where the last inequality follows from the fact that $z\mapsto \Psi(z;z)$ is $(1+\varepsilon)L$-Lipschitz continuous. Combining yields the one step improvement bound.

To prove (b), we first the quadratic contraction parameter using convexity. Observe that $0<\eta_{t}\leq(\gamma-\varepsilon L)(2(1+\varepsilon)^{2}L^{2})^{-1}$, implies that
    \begin{equation*}
        1 -2(\gamma-\varepsilon L)\eta_{t} + 2(1+\varepsilon)^{2}L^{2}\eta_{t}^{2} \leq 1-2(\gamma-\varepsilon L)\eta_{t}.
    \end{equation*}
Denoting $C=2(\gamma-\varepsilon L)$, and $\Delta = \bar{\nu}^{2}\eta^{2}2^{1+2\theta}$ it follows that 
    \begin{equation}
            \EX_{t}\Vert  z_{t+1} - \bar{z} \Vert^{2} 
            \leq (1-C\eta_{t}) \Vert z_{t}-\bar{z} \Vert^{2} + \eta_{t}^{2}\Delta.
    \end{equation} 
We proceed by induction. Clearly the bound in \cref{thm:stochastic_convergence} holds for $t=0$. Supposing it holds for $t$, we have that 
\begin{align*}
    \EX\norm{z_{t+1}-\bar{z}} & \leq \left( 1-\frac{C \ell}{\kappa + t} \right) \frac{\zeta}{\kappa+t} 
    + \frac{\Delta\ell^{2}}{(\kappa + t)^{2}} \\
    & \leq \frac{\kappa + t - 1 }{ ( \kappa + t )^{2} }\zeta - \frac{C\ell - 1}{( \kappa + t )^{2}}\zeta 
    + \frac{\Delta\ell^{2}}{(\kappa + t)^{2}} \\ 
    & \leq \frac{\kappa + t - 1 }{ ( \kappa + t )^{2} }\zeta \\
    & \leq \frac{\zeta}{( \kappa + (t+1) )^{2} },
\end{align*}
where the penultimate step follows from the fact that $(C\ell - 1)\zeta + \Delta\ell^{2}<0$.
\end{proof}

This concludes our analysis of equilibrium points. In the following section, we discuss how to compute saddle points.

\section{Saddle Points and Mixture Dominance}
\label{sec:saddle}
By introducing the equilibrium point problem, we have shifted the attention to a class of solutions that are less computationally burdensome to obtain while still serving as meaningful solutions within the context of decision-dependent stochastic problems. In this section, we demonstrate that finding saddle points is still possible for some  well-behaved distributional maps. In particular,  we consider a condition which we call \textit{opposing mixture dominance}. 
To outline the main arguments, we focus on the saddle-point problem~\eqref{problemStatement}. 
In following, we define the notion of opposing mixture dominance.

\begin{assumption}{\textit{\textbf{(Opposing Mixture Dominance)}}}
\label{as:mixDom}
For any $x,x',x_{0}\in\real^{d}$, $y,y',y_{0}\in\real^{n}$ and $\tau\in[0,1]$, the distributional map satisfies a \textit{convex shift} in $x$
\begin{equation*}
\label{xdom}
        \underset{w\sim D(\tau x + (1-\tau)x',y)}{\EX}[\phi(x_{0},y_{0},w) ]  \leq  \underset{w\sim \tau D(x,y) + (1-\tau)D(x',y)}{\EX} [\phi(x_{0},y_{0},w) ],
\end{equation*}
and \textit{concave shift} in $y$
\begin{equation*}
\label{ydom}
        \underset{w\sim \tau D(x,y) + (1-\tau)D(x,y')}{\EX} [\phi(x_{0},y_{0},w) ] \leq \underset{w\sim D(x,\tau y + (1-\tau)y')}{\EX}[\phi(x_{0},y_{0},w) ].
\end{equation*}

\end{assumption}

As an example, we show that  Bernoulli mixtures satisfies this assumption.

\begin{example}{\textit{\textit{(Bernoulli Mixtures)}}}
If the distributional map $D:\realDomain\rightarrow\mathcal{P}(M)$ is given by $D(x,y) =\text{Bernoulli}(p(x,y))$ where $p:\realDomain\rightarrow\real$ is the bilinear function 
$$p(x,y) = \langle x,Ay \rangle + \langle b,x \rangle + \langle c,y \rangle + d$$
 then \cref{as:mixDom} is satisfied since $D(\tau x + (1-\tau)x',y) = \tau D(x,y) + (1-\tau)D(x',y)$ and $\tau D(x,y) + (1-\tau)D(x,y') = D(x,\tau y + (1-\tau)y')$.
\end{example}

 \begin{example}{\textit{\textit{(Location-Scale Families)}}}
 \label{ex:Locationscale}
A distributional map $D:\realDomain\rightarrow\mathcal{P}(\real^{m})$ induces a location-scale family provided that for any $z\in\realDomain$, $w\sim D(z)$ if and only if $w\overset{d}{=} Aw_{0} + Bz + c $ where $w_{0}$ is some stationary zero-mean random variable. A sufficient condition for \cref{as:mixDom} to hold is that $\phi$ is convex in the random variable $w$. A detailed proof of this fact is provided in the next section.
 \end{example}

In the previous section, we made the assumption that our random variables are supported on some general Polish space and are induced by a Radon probability measure parameterized by $z=(x,y)\in\realDomain$. 
Here, we assume  without loss of generality that the distributional map induces a probability density function $p(w;x,y)$ and write the objective as 
    $\Phi(x,y) = \int_{M} \phi(x,y,w)p(w;x,y)dw$.
The analysis that follows is identical for the case when the density $p(w,;x,y)$ corresponds to discrete probability distribution parameterized by $(x,y)$ and the proofs follow \textit{mutatis mutandis}.

Below, we demonstrate that the opposing mixed dominance assumption is sufficient to guarantee that the objective is convex-concave in the distribution inducing arguments. The crux of this proof is observing that convex combinations of probability distributions have a density function defined by the convex combination of the underlying density functions. 

\begin{lemma}
Let \cref{as:mixDom} hold. Then, for any $z_{0}\in\realDomain$, the function $(x,y)\mapsto \EX_{w\sim D(x,y)}[\phi(z_{0},w)]$ is convex-concave on $\realDomain$.
\end{lemma}

\begin{proof} Fix $z_{0}\in\zDomain$, $x,x'\in\mathcal{X}$, and $y,y'\in\mathcal{Y}$ and let $\tau\in[0,1]$. Observe that since the distribution
$\tau D(x,y) + ( 1-\tau)D(x',y)$ is a convex mixture, then its probability density function is convex sum of the probability density functions for $D(x,y)$ and $D(x',y)$. That is, if $p_{\tau}$ is the density function for the convex mixture, and $p_{1}$ and $p_{2}$ are the density functions for $D(x,y)$ and $D(x',y)$, respectively, then $p_{\tau}(w) = \tau p_{1}(w) + (1-\tau)p_{2}(w)$. From this, we conclude that
\begin{equation*}
    \underset{w\sim\tau D(x,y) + ( 1-\tau)D(x',y)}{\EX}[\phi(z_{0},w)] \leq \tau \underset{w\sim D(x,y)}{\EX}[\phi(z_{0},w)] + (1-\tau) \underset{w\sim\tau D(x',y)}{\EX}[\phi(z_{0},w)].
\end{equation*}
Combining this with \cref{as:mixDom}, we get that 
\begin{equation*}
   \underset{w\sim D(\tau x  + (1-\tau)x',y)}{\EX}[\phi(z_{0},w)]\leq  \tau \underset{w\sim D(x,y)}{\EX}[\phi(z_{0},w)] + (1-\tau) \underset{w\sim\tau D(x',y)}{\EX}[\phi(z_{0},w)].
\end{equation*}
This proves convexity of $x\mapsto \EX_{w\sim D(x,y)}[\phi(z_{0},w)]$ for any $y$. The  concavity in $y$ can be shown using similar steps. 
\end{proof}

We can then utilize this result in conjunction with our previous assumptions to get strong-convexity-strong-concavity of the objective $\Phi$.

\begin{theorem}{(\textit{\textbf{Strong-Convexity-Strong-Concavity}})}
\label{thm:scscObj}
If \crefrange{as:scsc}{as:sensitive} and \cref{as:mixDom} hold, then $(x,y) \mapsto \Phi(x,y)$ is $(\gamma-2\varepsilon L)$-strongly-convex-strongly-concave over $\realDomain$.
\end{theorem}

\begin{proof}
We prove the assertion by first demonstration that strong-convexity holds in $x$ for $y$ fixed. Strong-concavity will follow similarly. By applying $\gamma$-strong-concavity of $\phi$ in $x$, we get that
\begin{multline}
\label{thm:scsc_one}
\Phi(x',y ; x' y) - \Phi(x,y;x',y) \geq 
 \langle x' - x, \underset{w\sim D(x',y)}{\EX}[\nabla_{x}\phi(x,y,w)]\rangle + \frac{\gamma}{2}\norm{x-x'}^{2}.
\end{multline}
By the $L$-smoothness of the gradient, we get that 
\begin{equation*}
 \langle x' - x, 
 \underset{w\sim D(x,y)}{\EX}[\nabla_{x}\phi(x,y,w)]-  \underset{w\sim D(x',y)}{\EX}[\nabla_{x}\phi(x,y,w)] 
 \rangle
 \leq \varepsilon L \norm{x-x'}^{2}
\end{equation*}
which is equivalent to 
\begin{equation}
\label{thm:scsc_two}
     0  \geq \langle x' - x,  \underset{w\sim D(x,y)}{\EX}[\nabla_{x}\phi(x,y,w)]-  \underset{w\sim D(x',y)}{\EX}[\nabla_{x}\phi(x,y,w)]  \rangle - \frac{2\varepsilon L}{2} \norm{x-x'}^{2}.
\end{equation}
Since for any $z_{0}\in\realDomain$ the function $(x,y)\mapsto \EX_{w\sim D(x,y)}[\phi(z_{0},w)]$ is convex-concave, we have that 
\begin{multline}
\label{thm:scsc_three}
    \Phi(x,y;x,y) - \Phi(x,y;x',y) \geq
     \langle x- x', \underset{w\sim D(x',y)}{\EX}[\phi(x,y,w)\nabla_{x}\log p(w;x,y)] \rangle
\end{multline}
by setting $z_{0} = (x,y)$. By adding inequalities \crefrange{thm:scsc_one}{thm:scsc_three} we obtain 
\begin{equation*}
\Phi(x',y) - \Phi(x,y) \geq 
\langle x' -x , \nabla_{x}\Phi(x,y) \rangle  + \frac{\gamma - 2\varepsilon L }{2}\norm{x-x'}^{2},
\end{equation*}
which is equivalent to strong-convexity in $x$. Proof of strong-concavity in $y$ follows similarly and it is omitted due to space limitations.
\end{proof}

\subsection{Location-Scale Families}
In this section, we are  interested in solidifying the claims made in Example~\ref{ex:Locationscale} on Location-scale families, which have seen much attention in the literature on decision-dependent distributions as it arises naturally in many common examples \cite{miller2021outside}. A formal definition is provided next. 

\begin{definition}{\textit{\textbf{(Location-Scale Family)}}}
The distributional map $D:\realDomain\rightarrow\mathcal{P}(\real^{m})$ forms a location-scale family provided that for every $z\in\realDomain$ and $w\sim D(z)$, $w \overset{d}{=} Aw_{0} + Bz + c$ where $w_{0}\sim D_{0}$. In this model, $D_{0}\in\mathcal{P}(\real^{m})$ is a zero-mean stationary distribution while $A_{0}\in\real^{m\times m}$, $B\in\real^{m\times(d+n)}$, and $c\in\real^{m}$ are model parameters.
\end{definition}

To demonstrate that Location-scale Families satisfy \cref{as:mixDom}, we introduce the notion of convex stochastic orders. This is an ordering of random variables induced by convex functions.

\begin{definition}{\textit{\textbf{(Convex Order)}}}\cite[Definition 7.A.1]{shaked2007stochastic}
If two m-dimensional random vectors $u$ and $w$ are such that $\EX[f(u)] \leq \EX[f(w)]$, 
for all convex functions $f:\real^{m}\rightarrow\real$, then we say that $u$ is less than $w$ in the convex order and write $u\leq_{cx}w$.
\end{definition}

Demonstrating an ordering from this definition alone proves difficult. Instead, we look to the following theorem that characterizes random variables in the convex stochastic order via couplings. 

\begin{theorem}\cite[Theorem 7.A.1]{shaked2007stochastic} 
\label{thm:coupling}
The random vectors $u\sim\mu$ and $w\sim\nu$ satisfy $u\leq_{cx}w$ if and only if there exists $\hat{u}\overset{d}{=}$u and $\hat{w}\overset{d}{=}w$ such that $\EX[\hat{w}\vert \hat{u}] = \hat{u} \ a.s.$
\end{theorem}

Following this characterization, we demonstrate that location-scale families have a special relationship between the convex-combination family and the corresponding convex-mixture. 

\begin{lemma}
\label{lem:equalInMean}
Let the distributional map $D:\real^{d}\times\real^{n} \rightarrow\mathcal{P}(\real^{m})$ be a location scale family. Then for any $z,z'\in\real^{d}\times\real^{n}$ and $\tau\in[0,1]$,
\begin{equation*}
    \underset{z\sim D(\tau z + (1-\tau)z')}{\EX} [f(w)] =  \underset{z\sim \tau D(z) + (1-\tau)D(z')}{\EX} [f(w)]
\end{equation*}
for any convex function $f:\real^{m}\rightarrow\real$.
\end{lemma}

\begin{proof}
Fix $\tau\in[0,1]$ and $z,z'\in\real^{d}\times\real^{n}$. In this proof, we use Theorem \ref{thm:coupling} to show that if $w\sim D(\tau z + (1-\tau)z')$ and $w'\sim \tau D(z) + (1-\tau)D(z')$, then we can define couplings that imply that  $w\leq_{cx} w'$ and $w'\leq_{cx} w$. To this end, a key observation is that, if we denote the discrete random variable $T$ as
\begin{equation*}
    T = 
    \begin{cases}
    z \ \text{w.p. $\tau$}, \\ 
    z' \ \text{w.p $1-\tau$},
    \end{cases}
\end{equation*}
then $w'\sim \tau D(z) + (1-\tau)D(z')$ if and only if $w\overset{d}{=}Aw_{0}+BT+c$. \\

First, we suppose that $w\sim D(\tau z + (1-\tau)z')$. Then let $w'\overset{d}{=} w - B(\tau z + (1-\tau)z') + BT$. It follows that  $\EX[w'\vert w] = w$, and $w'\overset{d}{=}Aw_{0}+BT+c$. Hence $w'\sim \tau D(z) + (1-\tau)D(z')$. This proves that $w\leq_{cx} w'$. 

Conversely, if we suppose that $w'\sim \tau D(z) + (1-\tau)D(z')$ and set $w\overset{d}{=} w' + B(\tau z + (1-\tau)z') - BT$ then $w'\leq_{cx} w$ follows. The statement follows from the definition of the convex order. 
\end{proof}

Since this Lemma holds for any convex function $f$, it holds for stochastic payoff $\phi$ provided that it is convex in $w$. This combined with the fact that Location-Scale Families are $\varepsilon$-Lipschitz with $\varepsilon = \norm{B}_{2}$ is sufficient for $\Phi$ to be strongly-convex-strongly-concave. 

\begin{theorem} 
Suppose that $\phi$ satisfies \cref{as:scsc,as:jointSmooth}, and the constraint sets $\mathcal{X}$ and $\mathcal{Y}$ satisfy \cref{as:sets}. If $D$ if a location-scale family and $\phi$ is convex in $w$, then $\Phi$ is $(\gamma-2\varepsilon L)$- strongly-convex-strongly-concave. 
\end{theorem}

\begin{proof}
The proof amounts to demonstrating that $D$ being a location-scale family and $\phi$ being convex in $w$ is sufficient to satisfy \cref{as:mixDom} and \cref{as:sensitive}. The result then follows by \cref{thm:scscObj}.We observe that \cref{lem:equalInMean} implies that \cref{as:mixDom} holds. As for $D$ being $\varepsilon$- Lipschitz, \cref{as:sensitive}, we claim that $W_{1}(D(z),D(z')) \leq \norm{B}_{2} \norm{z-z'}$.
Then the Assumption holds with $\varepsilon=\norm{B}_{2}$. By definition, 
\begin{equation*}
    W_{1}(D(z),D(z')) = \inf_{\Pi(D(z),D(z'))} \underset{(w,w')\sim \Pi(D(z),D(z'))}{\EX}[\norm{w-w'}_{2}]
\end{equation*}
where the infimum is taken over all couplings of the distributions $D(z)$ and $D(z')$. We find that if $w_{0}\sim D_{0}$, then setting $w \overset{d}{=} Aw_{0} + Bz + c$ and  $w' \overset{d}{=} Aw_{0} + Bz' + c$ implies that $w\sim D(z)$ and $w'\sim D(z')$ and $\norm{w-w'} = \norm{B(z-z')}$. Thus, the result follows.
\end{proof}


\subsection{A Zeroth-Order Algorithm}

In this section we consider the use of a zeroth-order algorithm, which we refer to as DFO, where a stochastic gradient estimator is built only using function evaluations. This algorithm is suitable in the setting where opposing mixture dominance in \cref{as:mixDom} is known to hold, but a model for the distributional map is not available. The use of zeroth-order algorithms has been studied extensively within the context of derivative free games in \cite{bravo2018bandit, drusvyatskiy2021improved}. 

Denote $\mathbb{B}_{k}$ and $\mathbb{S}_{k}$ as the uniform distributions over the unit ball, $\mathcal{B}_{k}=\{x\in\real^{k} \vert ~ \norm{x}\leq 1 \}$
and unit sphere,
$\mathcal{S}_{k}=\{x\in\real^{k} \vert ~ \norm{x}= 1 \},$
in $\real^{k}$ respectively. Additionally, denote $\mathbb{S}$ and $\mathbb{B}$ as joint distributions such that $v=(v_{1},v_{2})\sim \mathbb{B}$, $u = (u_{1},u_{2})\sim\mathbb{S}$ with $v_{1}\sim\mathbb{B}_{d}$, $v_{2}\sim\mathbb{B}_{n}$ and  $u_{1}\sim\mathbb{S}_{d}$, $u_{2}\sim\mathbb{S}_{n}$. The algorithmic map is then given by
\begin{equation}
\label{eq:zero_order_map}
    \mathcal{F}_{t}^{\delta}(z) = \Pi_{(1-\delta)\mathcal{Z}} \left(z-\eta_{t} \Omega_{\delta}(z) \right) 
\end{equation}
for $\eta_{t}>0$, with zeroth-order gradient map
\begin{equation}
\label{eq:zero_order_gradient}
    \Omega_{\delta}(z) = \left( 
    \frac{d}{\delta}\phi(z+\delta u, w)u_{1}, 
     \ -\frac{n}{\delta}\phi(z+\delta u, w)u_{2} \right)
\end{equation}
where $\delta > 0$, and $u=(u_{1},u_{2})$ with $u_{1}\sim \mathbb{S}_{d}$ and $u_{2}\sim \mathbb{S}_{n}$. Note that by projecting onto the restricted set $(1-\delta)\mathcal{Z}$ we  retain feasibility throughout the iterations of the algorithm. Since we evaluating the stochastic objective at points perturbed by vectors on the unit sphere, we must introduce an additional assumption to ensure that the domain of our function is appropriate.

\begin{assumption}
There exist positive radii $r,R>0$ such that $\mathcal{Z}$ satisfies $r\mathbb{B}_{d+n}\subseteq \mathcal{Z}\subseteq R\mathbb{B}_{d+n}$. 
\end{assumption}

The gradient estimator in \cref{eq:zero_order_gradient} naturally arises when considering the smoothed objective over the unit ball, given by 
\begin{equation}
    \Phi_{\delta}(z) = \underset{v\sim\mathbb{B}}{\EX} [ \Phi(z+\delta v) ] = \underset{v\sim\mathbb{B}}{\EX} \left[ \underset{w\sim D(z+\delta v)}{\EX}\left[\phi(z+\delta v,w) \right]\right]
\end{equation}
and its associated gradient map $\Psi_{\delta}(z) = \left( \nabla_{x}\Phi_{\delta}(z), -\nabla_{y}\Phi_{\delta}(z) \right) $. These together form the perturbed saddle point problem  
\begin{equation}
    \min_{x\in(1-\delta)\mathcal{X}}     \max_{y\in(1-\delta)\mathcal{Y}} \Phi_{\delta}(x,y),
\end{equation}
whose solutions we will we denote $z_{\delta}^{*}=(x_{\delta}^*, y_{\delta}^*)$.
It follows that $\Omega_{\delta}$ is an unbiased estimator of this gradient map, and hence it will allow us to find saddle points without requiring more information about the objective or distributional map. We formalize this in the following.

\begin{lemma}{\textit{\textbf{(Gradient Estimator)}}} If $\delta > 0$, then $\EX_{u\sim \mathbb{S}}[ \EX_{w\sim D(z)} \Omega_{\delta}(z)] = \Psi_{\delta}(z)$, for all $z\in\realDomain$.
\end{lemma}
Proof of this result follows from \cite[Lemma C.1]{bravo2018bandit}. The the fact that we can estimate the gradient map using only a single function evaluation is an attractive feature of~\eqref{eq:zero_order_map}. There are alternatives multi-point estimators that use more function evaluations, but since the expectation in our problem also depends on the decision variables, they are biased. Furthermore, in the following we show that the considered perturbed gradient map retains strong-monotonicity.

\begin{lemma}{\textit{\textbf{(Strong Monotonicity)}}}
If the gradient of the objective $\Phi$, given by $\Psi(z) = \left( \nabla_{x}\Phi(z), -\nabla_{y}\Phi(z) \right) $ is $(\gamma-2\varepsilon L)$-strongly-monotone, then $\Psi_{\delta}$ is $(\gamma-2\varepsilon L)$-strongly-monotone for any $\delta>0$.

\end{lemma}

Indeed, by perturbing the objective and the constraint set by $\delta$, the solution of the perturbed saddle point problem will may $\textit{may}$ be different from the solutions of the original problem. In the following, we bound the discrepancy between solutions.  

\begin{lemma}(\textbf{\textit{Bounded Approximation}}) If $\delta < r$, and $\Psi$ is $(\gamma -\varepsilon L)$ - strongly monotone, then
\begin{equation}
\label{lem:delta_approx}
    \norm{z^{*} - z^{*}_{\delta}} \leq \delta\left( \left(1 + \frac{\sqrt{2L}}{(\gamma-2\varepsilon L)}\right)\norm{z^{*}} + \frac{2L}{(\gamma-2\varepsilon L)} \right).
\end{equation}
\end{lemma}

Finally, we are ready to demonstrate the performance of the algorithm. Here we impose the additional restriction that $\delta$ may not exceed the radius of the largest ball completely contained in $\mathcal{Z}$, which we denoted as $r$.

\begin{theorem}{\textbf{\textit{(Convergence to the Perturbed Solution)}}}
\label{thm:saddle_nbhd_convergence}
Suppose that $\delta\leq r$ and $\eta_{t}=\ell(\kappa+ t)^{-1}$ for $\ell > (2(\gamma - 2\varepsilon L))^{-1}$ $\kappa >0$. Then, the sequence of iterates $\{z_{t}\}_{t\geq 0}$ generated by the derivative free stochastic method satisfy
\begin{equation}
\label{thm:approx_saddle_point}
    \EX\norm{z_{t}-z_{\delta}^{*}}^{2} \leq \frac{\zeta}{\kappa + t}, \quad \text{where} \quad \zeta : = \max
    \left\{
    \kappa\norm{z_{0}-z_{\delta}^{*}}^{2}, \ \frac{B^{2}(n^{2}+m^{2})\ell^{2}}{\delta^{2}(2(\gamma-2\varepsilon L)\ell-1)}
    \right\}
\end{equation}
where $B = \max_{z\in\mathcal{Z}, w\in M} \vert \phi(z,w)\vert $. 
\end{theorem}
\begin{proof}
For notational convenience, we write $\hat{\gamma} = \gamma-2\varepsilon L$, and $C=B^{2}(n^{2}+m^{2})\delta^{-2}$. By applying non-expansiveness of the projection map, we get 
\begin{align*}
    \EX_{t}\norm{z_{t+1}-z_{\delta}^{*}}^{2} 
    & \leq \EX_{t}\norm{z_{t}-z_{\delta}^{*}}^{2} - 2\eta_{t}\EX_{t}\langle z_{t} - z_{\delta}^{*}, \Omega_{\delta}(z_{t}) \rangle + \EX_{t} \norm{\Omega_{\delta}(z_{t})}^{2} \\
    & \leq \norm{ z_{t}-z_{\delta}^{*} }^{2} - 2\hat{\gamma}\eta_{t} \norm{z_{t}-z_{\delta}^{*}}^{2} + C\eta_{t}^{2}  \\
    & = (1-2\hat{\gamma}\eta_{t})\norm{z_{t}-z_{\delta}^{*}}^2 + C\eta_{t}^{2}.
\end{align*}
In substituting the step size $\eta_{t}=(\ell\gamma(\kappa+ t))^{-1}$, we find that 
\begin{equation*}
     \EX_{t}\norm{z_{t+1}-z_{\delta}^{*}}^{2} \leq \frac{\kappa+t-2\hat{\gamma}\ell}{\kappa+t}\norm{z_{t}-z_{\delta}^{*}}^{2} + \frac{C}{(\kappa + t)^{2}}.
\end{equation*}
As in the proof of \cref{thm:second_moment_convergence}, the result follows by induction. 
\end{proof}

This concludes our proof of convergence to the perturbed saddle point $z_{\delta}^{*}$. Obtaining convergence to the saddle point $z^{*}$ is a matter of applying the stochastic algorithm in stages with a geometrically decaying step size.

\section{Numerical Experiments for Electric Vehicle Charging}
To illustrate our results, we apply our algorithms to the electric vehicle charging problem outlined in \cref{subsec:electric_vehicles}. This is a competition between two providers in which each provider seeks to maximize their  profit.

\begin{figure}[t]
\label{fig:experiments}
\centering
 \subfloat[]{\label{fig:a}\includegraphics[width = 0.50\textwidth]{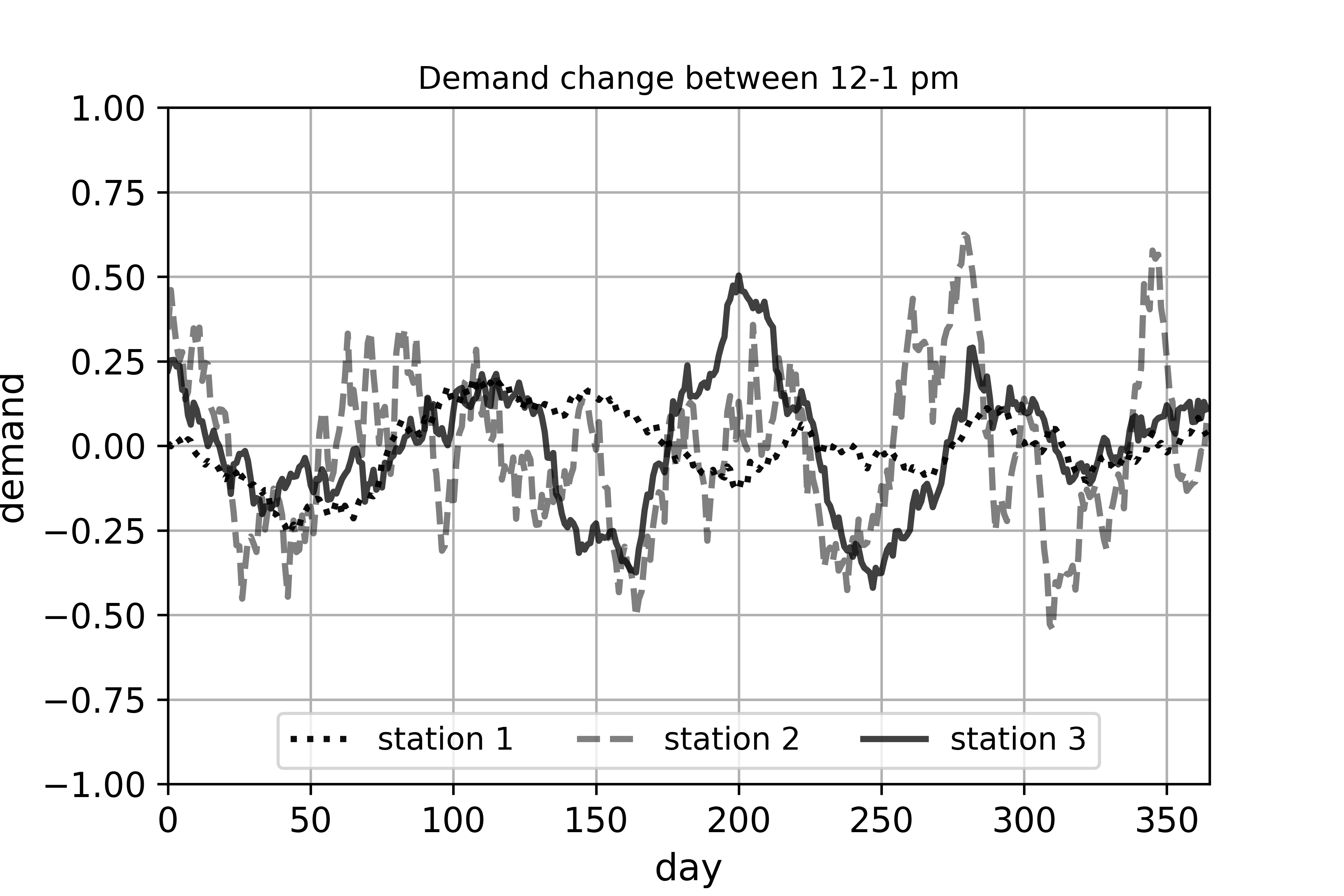}}
\subfloat[]{\label{fig:b}\includegraphics[width =0.50\textwidth]{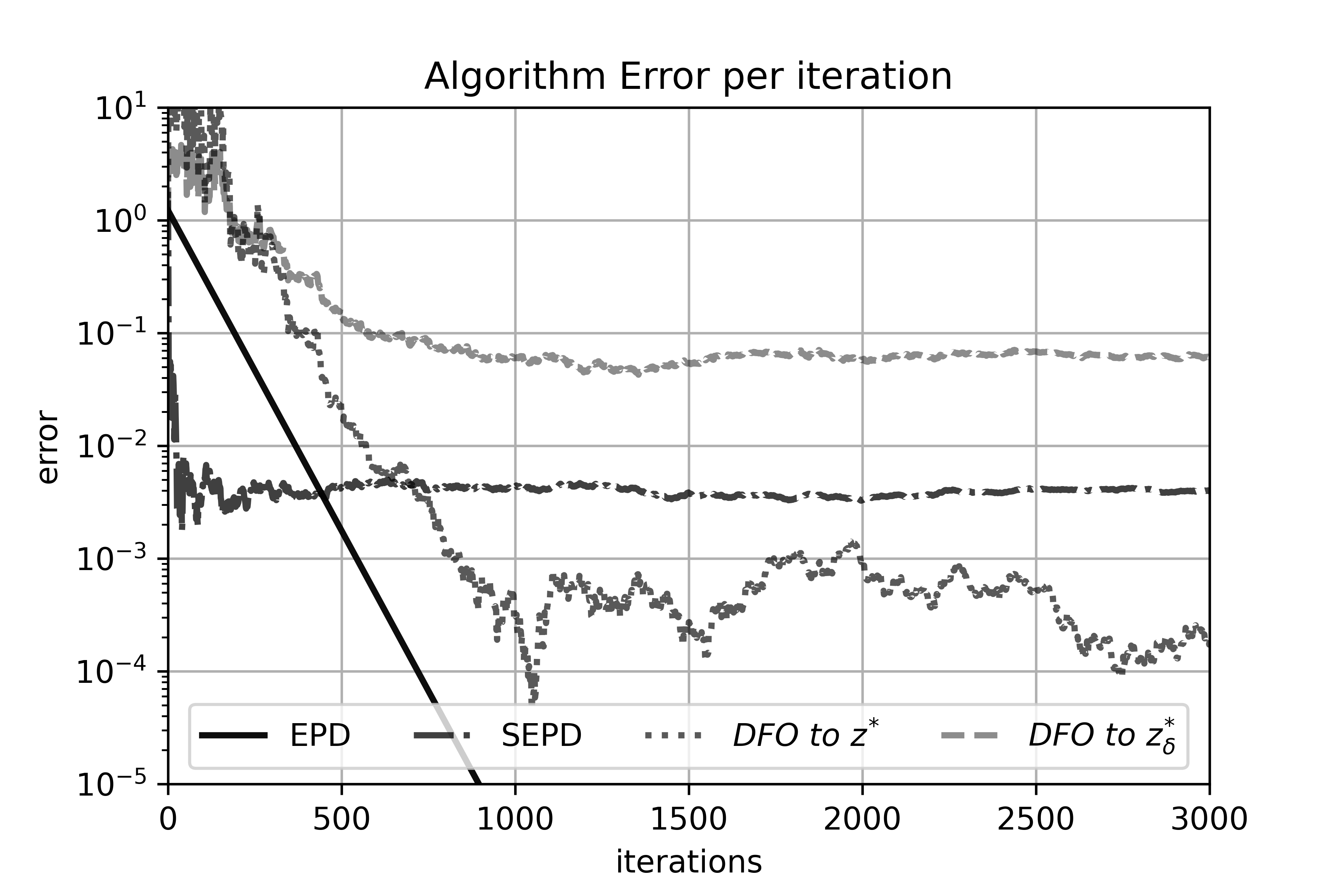}} 
\caption{Data and results from numerical experiments. In (a) deviation in average demand for provider one's stations between 12 and 1 pm over 365 days.(b) the error of each algorithm depicted over 3,00 iterations. Error of the derivative-free method is depicted in both distance to the saddle point $z^{*}$ as well as distance to the perturbed saddle point $z_{\delta}^{*}$.}
\label{fig:testfig}
\vspace{-.5cm}
\end{figure}

In our simulation, each provider has access to three distinct regions, each of which having one station. The demand for each station is dictated by the data distributions from \cite{gilleran2021electric}. Each station is comprised of $50$, $150$, or $350$ kW chargers with either $2$ or $6$ ports. We randomize this allocation at initialization. Data is processed by averaging the demand over each hour-long time window. After picking an hour block, we re-scale the data by subtract the mean and dividing by the variance. We choose the demand change in the 12-1pm block, and depict data for the year in \cref{fig:a}. Our simulations use charging utility values of $\gamma_{j,i}=1$ for $j\in[2], i\in[3]$, elasticity values of $(A_{1})_{i,j} = (-0.3)\delta_{i,j}$, $(A_{2})_{i,j} = (0.3)\delta_{i,j}$, $B_{1} = A_{2}$, and $B_{2} = A_{1}$, and location utility values $r_{i}=0$ for each station. The price deviations $x$ and $y$ are restricted to the interval $[-1,2]$ for each station, representing a nominal price of $\$1$ and a maximum price change of twice the nominal price. Hence $\mathcal{X}=\mathcal{Y}=[-1,2]^{3}$.

We run each algorithm for 10,000 iterations, and depict the first 3,000 iterations in \cref{fig:b} to provide a side-by-side comparison. The equilibrium points and saddle points are computed via primal-dual with constant step size $\eta=0.001$ as a means to compute the norm squared errors $\norm{z_{t}-\bar{z}}^{2}$ and $\norm{z_{t}-z^{*}}^{2}$. We run SEPD and the zeroth order algorithm with the polynomial decay step-size schedules described in \cref{thm:stochastic_convergence} and \cref{thm:saddle_nbhd_convergence}. In the latter, we choose a fixed $\delta$ value of $0.05$. Relative to EPD, our results for these stochastic algorithms only guarantee sub-linear convergence at best; the step-size effectively converges to zero faster than the error resulting in the plateau of our error curves. The python code is publicly available\footnote{\url{https://github.com/killianrwood/charging-market}}.

\section{Concluding Remarks}

The paper focused on stochastic saddle point problems with decision-dependent distributions. We introduced the notion of equilibrium points and provide conditions for their existence and uniqueness. We showed that the distance between the two classes of solutions is bounded provided that the objective has a strongly-convex-strongly-concave payoff and Lipschitz continuous distributional map. We developed and analyzed deterministic and stochastic primal-dual algorithms.  In particular, using a sub-Weibull model for the errors emerging in the gradient computation, we provided error bounds in expectation and in high probability that hold for each iteration; we also showed convergence to a neighborhood in expectation and almost surely. Finally, we investigate an  opposing mixture dominance condition that ensures the objective is strongly-convex-strongly-concave, and we focused on a zeroth-order algorithm. 

Future directions  will explore proximal-based methods for problems with a more general geometry. We will also consider  alternative derivative free methods  and approaches that incorporate gradient information by estimating the distributional map. 

\printbibliography

\end{document}